\documentclass[letterpaper, 10 pt]{IEEEtran}  

\IEEEoverridecommandlockouts                              

\usepackage{amsmath}
\usepackage{amssymb}
\usepackage{amstext}

\usepackage{graphicx}

\usepackage{subcaption}

\usepackage{balance}

\usepackage{listings}
\usepackage{enumerate}

\usepackage[english]{babel}
\usepackage[utf8]{inputenc}

\usepackage[dvipsnames]{xcolor}

\usepackage{tikz}
\usetikzlibrary{arrows}
\usetikzlibrary{positioning}

\newtheorem{theorem}{Theorem}
\newtheorem{corollary}{Corollary}
\newtheorem{lemma}{Lemma}
\newtheorem{proposition}{Proposition}
\newtheorem{remark}{Remark}

\newtheorem{example}{Example}

\numberwithin{theorem}{section}
\numberwithin{corollary}{section}
\numberwithin{lemma}{section}
\numberwithin{proposition}{section}
\numberwithin{remark}{section}
\numberwithin{definition}{section}
\numberwithin{example}{section}

\newcommand{\mR}{\mathbb{R}}

\newcommand{\mC}{\mathbb{C}}
\newcommand{\mD}{\mathbb{D}}

\newcommand{\mT}{{\mathbb T}}
\newcommand{\mM}{{\mathbb M}}

\newcommand{\gln}{{\mathbb{GL}_n}}

\newcommand{\mU}{{\mathbb U}}
\newcommand{\mH}{{\mathbb H}}
\newcommand{\mP}{{\mathbb P}}
\newcommand{\mA}{{\mathbb A}}

\def\RR{\boldsymbol{\rm R}}
\newcommand{\Linfty}{\boldsymbol{\rm L}_{\infty}}
\newcommand{\Hinfty}{\boldsymbol{\rm H}_{\infty}}

\let\odlangle\angle  
\renewcommand{\angle}{\odlangle \,}

\DeclareMathOperator*{\diag}{diag}

\DeclareMathOperator*{\real}{real}

\DeclareMathOperator*{\range}{range}

\DeclareMathOperator*{\esssup}{ess\,sup}

\DeclareMathOperator*{\bdry}{\partial}
\DeclareMathOperator*{\intr}{int}
\DeclareMathOperator*{\rank}{rank}

\newcommand{\homoparam}{\alpha}

\newcommand{\phimax}{\overline{\phi}}
\newcommand{\phimin}{\underline{\phi}}

\newenvironment{bsmallmatrix}
  {\left[\begin{smallmatrix}}
  {\end{smallmatrix}\right]}

\makeatletter
\def\iddots{\mathinner{\mkern1mu\raise\p@
\vbox{\kern7\p@\hbox{.}}\mkern2mu
\raise4\p@\hbox{.}\mkern2mu\raise7\p@\hbox{.}\mkern1mu}}
\makeatother

\graphicspath{{./}, {./figures/}}

\title{Gain and phase type multipliers\\ for feedback robustness}

\author{Axel Ringh, Xin Mao, Wei Chen, Li Qiu, and Sei~Zhen~Khong,
\thanks{*
This work was partially supported by the Knut and Alice Wallenberg foundation via grant KAW 2018.0349, by the Wallenberg AI, Autonomous Systems and Software Program (WASP) funded by the Knut and Alice Wallenberg Foundation, by the Hong Kong Research Grants Council under projects GRF~16200619 and GRF~16201120, by the Guangdong Science and Technology Department under the project No. 2019B010117002, by the National Natural Science Foundation of China under grants 62073003 and 72131001, and by the National Science and Technology Council of Taiwan under grants 113-2222-E-110-002-MY3 and 114-2218-E-007-011-. Corresponding author: Sei Zhen Khong.}
\thanks{A.~Ringh is with Department of Mathematical Sciences, Chalmers University of Technology and University of Gothenburg, SE-412 96 Gothenburg, Sweden. Email: {\tt\small axelri@chalmers.se}}
\thanks{X.~Mao is with School of Data Science and Society, University of North Carolina at Chapel Hill, Chapel Hill, NC 27599, USA. Email: {\tt \small xinm@unc.edu}}
\thanks{W.~Chen is with School of Advanced Manufacturing and Robotics \& State Key Laboratory for Turbulence and Complex Systems, Peking University, Beijing 100871, China. Email: {\tt\small w.chen@pku.edu.cn}}
\thanks{L.~Qiu is with School of Science and Engineering, The Chinese University of Hong Kong, Shenzhen, Guangdong, China. Email: {\tt\small qiuli@cuhk.edu.cn}}
\thanks{S.~Z.~Khong is with Department of Electrical Engineering, National Sun Yat-sen University, Kaohsiung 80424, Taiwan. Email: {\tt\small  szkhong@mail.nsysu.edu.tw}}
}

\begin{document}

\maketitle
\thispagestyle{empty}
\pagestyle{empty}

\begin{abstract}
It is known that the stability of a feedback interconnection of two linear time-invariant systems implies that the graphs of the open-loop systems are quadratically separated. This separation is defined by an object known as the multiplier. The theory of integral quadratic constraints shows that the converse also holds under certain conditions. This paper establishes that if the feedback is robustly stable against certain structured uncertainty, then there always exists a multiplier that takes a corresponding form. In particular, if the feedback is robustly stable to certain gain-type uncertainty, then there exists a corresponding multiplier that is of phase-type, i.e., its diagonal blocks are zeros. These results build on the notion of phases of matrices and systems, which was recently introduced in the field of control. Similarly, if the feedback is robustly stable to certain phase-type uncertainty, then there exists a gain-type multiplier, i.e., its off-diagonal blocks are zeros. The results are meaningfully instructive in the search for a valid multiplier for establishing robust closed-loop stability, and cover the well-known small-gain and the recent small-phase theorems.
\end{abstract}


\section{Introduction}

One of the most fundamental problems in control theory is feedback stability analysis.
In this context,
it is well known that topological graph separation is both necessary and sufficient for the stability of a well-posed feedback 
configuration~\cite{teel1996input, teel1996graphs}. Such topological graph separation is required to hold in the \emph{hard}
(a.k.a. \emph{unconditional}~\cite{megretski2010integral}) manner, i.e., the integrals involved are taken over $[0, T]$ for all $T > 0$. A specific
type of separation, called quadratic graph separation with linear multipliers, has been studied extensively in the nonlinear~\cite{zames1966input,
  zames1968stability, megretski1997system, lessard2016analysis, hu2016exponential,  khong2021iqc, khong2024connections} and linear~\cite{cantoni2011robustness, cantoni2013robust, khong2016robust, khong2025feedback} literatures. 
Quadratic graph separation has often been employed in the \emph{soft} (a.k.a. \emph{conditional}) manner, where the integrals are taken over
$[0, \infty)$ in conjunction with homotopies that are continuous in the graph topology. In the linear time-invariant (LTI) setting, soft quadratic
graph separation is equivalent, via the Parseval-Plancherel theorem, to two complementary frequency-domain
inequalities~\cite{megretski1993power}. Such inequalities are the main object of study in this paper.

In robust stability analysis, the objective is to determine if a feedback interconnection between a nominal system and a set of uncertainties is stable or not for all uncertainties in the set \cite[Chp.~9]{zhou1996robust}.
One way to guarantee stability is by finding a certificate that the graph of the nominal system is separated from the graph of each of the system in the uncertainty set.%
\footnote{Sometimes, the problem is better formulated in terms of a feedback interconnection between two sets of uncertain systems, see, e.g., \cite{georgiou1990optimal, qui1992feedback, zhao2020stabilization}. In this case, the problem becomes to verify that for each pair of systems from the two uncertainty sets, the graphs are separated. The results in this paper can be interpreted as any of these two cases.}
In the case of quadratic graph separation, the object of interest is a function of an LTI object known as a multiplier, and the search for a suitable multiplier for characterizing the uncertainty is a common theme in the vast literature on robust control, see, e.g., \cite{zames1968stability, willems1968some, megretski1997system, jonsson2000optimization, scherer2001lpv, kao2007stability, pfifer2015integral}.
While quadratic graph separation with linear multipliers has in general been used as a sufficient condition for robust stability, the chief focus of this paper is on the necessity of it.
Some elegant results along this direction have been obtained in \cite{iwasaki1998well}, where it was shown that the closed-loop stability of the interconnection between a matrix and a set of matrices is equivalent to the existence of a multiplier by which quadratic separation holds. In other words, quadratic graph separation is both necessary and sufficient for the closed-loop stability of matrices. Moreover, results along this direction also generalize to LTI systems \cite{iwasaki1998well}. 

This paper strengthens the existing results by revealing a number of intricate relationships between the type of feedback robustness and the structure of any multiplier needed to establish such a robustness. Specifically, we define multipliers of the gain type (a.k.a. magnitudinal
multipliers) to be multipliers whose off-diagonal blocks are 0, and show that the existence of a gain type multiplier is equivalent with that the closed-loop system is robust against phasal uncertainties, i.e., multiplication by arbitrary stable unitary (i.e., all-pass) transfer functions. On the other hand, multipliers of the phase type (a.k.a. phasal multipliers) are defined as multipliers that are $0$ on the diagonal blocks, and we show that the existence of a phase type multiplier is equivalent with that the closed-loop system is robust against magnitudinal uncertainties, i.e., arbitrary nonnegative scalings.
The novelty in these equivalences lies in the necessity-part:
if a feedback system is robust against phasal (resp. magnitudinal) uncertainties, then its robust stability can always be established using a multiplier of the gain (resp. phase) type.
The results are of both theoretical and practical interest: theoretical, since they reveal a fundamental connection between the structure of the uncertainties and the structure of the multiplier; and practical, since they imply that if a feedback system is expected to be robust against a certain form of uncertainties, then the search for a suitable multiplier to establish its robust stability can be restricted to one that admits a prescribed structure, and vice versa.

There are also other methods to determine robust stability in the case of structured uncertainty.
One of the most prominent ways is via the structural singular value, $\mu$, in which one considers block-diagonal type uncertainties \cite{doyle1982analysis, safonov1982stability, packard1993complex}.
While computing the value of $\mu$ exactly is in general a difficult problem (NP-hard) \cite{nemirovskii1993several, braatz1994computational}, the celebrated $(D,G)$-scaling is a computable upper-bound for the structured singular value \cite{fan1991robustness, packard1993complex, meinsma1997dual}.
Moreover, for (among other cases) robust stability with respect to scalar gain uncertainties \cite{fan1991robustness, meinsma1997dual} and to scalar phase uncertainties \cite{chellaboina2008structured}, the $(D,G)$-scaling bound is know to be tight.
It has also been shown that this upper-bound being less than one, which is a sufficient condition for robust stability of the interconnection, is equivalent with the existence of a multiplier for characterizing the loop (transfer) matrix~\cite{chou1999stability} (see also \cite{hall1993mixed, how1993connections, fu1997improved}).
In this work, we, among other things, also consider the case of scalar gain and phase uncertainties (see Theorems~\ref{thm:H_sectorial}, and~\ref{thm:multiplier_phase_uncertainty} for the matrix cases).
While the structure of the multipliers in the case of scalar phase uncertainties can be obtained by carefully analyzing and using the results in \cite{chellaboina2008structured} on the loop (transfer) matrix, we use a different approach to prove our main results, working directly with the multipliers and characterizing both the potentially uncertain open-loop (transfer) matrices as opposed to only the loop (transfer) matrix.
Moreover, although the $(D,G)$-scaling method is closely related to our work, for the three other types of structured uncertainties considered (see Theorems~\ref{thm:H_sectorial}, \ref{thm:TATstarSBSstar}, and~\ref{thm: UAVB}, for the matrix cases), existing results related to the structured singular value can, to the best of our knowledge, not be used to establish the necessity of the form of the multipliers for robust stability.

Finally, note that there exist relevant converse quadratic separation results that are different from those examined in this paper.
Such results typically state that a feedback system is robustly stable against an arbitrary uncertainty characterised by a quadratic constraint if and only if the
other open-loop subsystem satisfies the reverse quadratic constraint~\cite{khong2018converse, khong2021converse, khong2022converse, khong2023converse}. However, in these references the multiplier defining the quadratic constraint is explicitly specified, whereas in this work certain forms of feedback stability are shown to imply the existence of a multiplier by which quadratic graph separation of the open-loop systems is defined.

The outline of the paper is as follows: in Section~\ref{sec:background} we introduce necessary background material related to quadratic graph separation and its use in stability analysis of multiple-input-multiple-output (MIMO) LTI systems, and to sectorial matrices and phases of a matrix.
In Section~\ref{sec:phase_multipliers} we analyze the form of multipliers needed in order to guarantee robust stability with respect to certain types of gain uncertainties. The conclusion
is that the existence of certain types of phasal multipliers is a both necessary and sufficient condition. Similarly, Section~\ref{sec:gain_multipliers} is devoted to stability against certain types of phase uncertainties, and the existence of certain types of magnitudinal multipliers turns out to be a both necessary and sufficient condition.
In Section~\ref{sec:num_ex}, we use a numerical example to illustrate how the results of the paper can be used.
The main part of the paper ends with Section~\ref{sec:conclusions}, where we draw some conclusions. Finally, in order to improve the readability, some of the lengthier proofs are
deferred to appendices in the end of the paper.

\section{Background and notation}\label{sec:background}

In this section we present some background material on quadratic graph separation, transfer matrices and multipliers for feedback stability of LTI systems, and sectorial matrices and matrix phases. Moreover, the section is also used to set up the notation; basic notation is introduced in the paragraph below, and further notation is introduced where needed.

\paragraph*{Notation}
Let $j$ denote the imaginary unit, i.e., $j^2 = -1$.
For two sets $\mathcal{A}$, $\mathcal{B}$, let $\mathcal{A} \cup \mathcal{B}$ denote the union, let $\mathcal{A} \cap \mathcal{B}$ denote the intersection, and let $\mathcal{A} \setminus \mathcal{B}$ denote the set-difference, i.e., $\mathcal{A} \setminus \mathcal{B} := \{ a \in \mathcal{A} \mid a \not \in \mathcal{B} \}$.
Let $\mR$ and $\mC$ denote the real and complex numbers, respectively, $\mR^n$ and $\mC^n$ the real and complex vectors of length $n$, respectively, 
$\mR_{+} := [0, \infty)$, $\mR_{-} := (-\infty, 0]$, and $\mR_{--} = \mR_{-} \setminus \{ 0 \}$ the positive, negative, and strictly negative real numbers,
$\mC_+ := \{ z \in \mC \mid z = a+ jb, a > 0 \}$ the open right-half complex plane,
$\mT := \{ z \in \mC \mid |z| = 1 \}$ the unit circle,
and $\mD := \{ z \in \mC \mid |z| < 1 \}$ the open unit disc.
Next, let $\mM_{n,m}$ denote the set of complex matrices with $n$ rows and $m$ columns; for square matrices we simply write $\mM_n$.
Let $\gln \subset \mM_n$ denote the set of invertible matrices, $\mH_n \subset \mM_n$ the set of Hermitian matrices, $\mP_n \subset \mH_n$ the set of (Hermitian) positive definite matrices, and $\mU_n \subset \gln$ the set of unitary matrices.
For the corresponding sets of matrices with real entries, we write $\mM_{n,m}(\mR)$, etc.  Moreover, on the set of Hermitian matrices we use $\succeq$ to denote the Loewner partial order, i.e., for $H_1, H_2 \in \mH_n$, $H_1 \succ H_2$ and $H_1 \succeq H_2$ means that $H_1 - H_2$ is positive definite and positive semi-definite, respectively; see, e.g. \cite[Sec.~7.7]{horn2013matrix}.
Furthermore, by $\cdot^T$ and $\cdot^*$ we denote the transpose and the conjugate transpose of a matrix, respectively, and two matrices $A,B \in \mM_n$ are said to be congruent if there exists a $C \in \gln$ such that $A = C^* B C$.
By $I_{n}$ we denote the identity matrix of size $n \times n$; sometimes the subscript $n$ is omitted when the dimension is clear from the context.
Finally, $\lambda(\cdot)$ denotes the set of eigenvalues, and $\sigma(\cdot)$ denotes the set of singular values of a matrix, i.e., for a matrix $A \in \mM_{n,m}$, $\sigma_i(A) = \sqrt{\lambda_i(A^*A)}$ and hence $A$ has $m$ singular values. By convention, the singular values are sorted in a nonincreasing order, and if $m > n$ this means that $\sigma_{n + \ell} (A) = 0$ for $\ell = 1, \ldots, m-n$.

\begin{figure}[tb]
\begin{center}
  
  \tikzstyle{int}=[draw, minimum size=2em]
  \tikzstyle{init} = [pin edge={to-,thin,black}]
  \tikzstyle{block} = [draw, rectangle, 
    minimum height=3em, minimum width=6em]
  \tikzstyle{sum} = [draw, circle]
  \tikzstyle{input} = [coordinate]
  \tikzstyle{output} = [coordinate]
  \tikzstyle{pinstyle} = [pin edge={to-,thin,black}]
  \tikzstyle{dott} =[circle,fill,inner sep=1.5pt]
  
  \begin{tikzpicture}[auto, node distance=2cm,>=latex', color=black]
    \node [sum, name=upperleft] {};
    \node [block, right of=upperleft] (disturbance) {$B$};
    \node [input, right of=disturbance] (upperright) {};
    \node [sum, below of=upperright] (lowerright) {};
    \node [block, left of=lowerright] (system) {$A$};
    \node [input, below of=upperleft] (lowerleft) {};
    
    \node[input, left = 3em of upperleft] (upperleftleft) {};
    \node[input, right = 3em of lowerright] (lowerrightright) {};

    \node [above right = 4pt of upperleftleft] (pointb) {$p$};
    \node [below left = 6pt of lowerrightright] (pointa) {$q$};

    \node [above left = 6pt of lowerleft] (pointc) {$r$};
    \node [below right = 6pt of upperright] (pointd) {$w$};
    
    \node [above left = 16pt of lowerright] (minusTMP) {};
    \node [below = 2pt of minusTMP] (minus) {$-$};

    \draw [<-] (disturbance) -- (upperright);
    \draw [-] (upperright) -- (lowerright);
    \draw [<-] (lowerright) -- (system);
    \draw [<-] (system) -- (lowerleft);
    \draw [-] (lowerleft) -- (upperleft);
    \draw [<-] (upperleft) -- (disturbance);
    
    \draw [->] (upperleftleft) -- (upperleft);
    \draw [->] (lowerrightright) -- (lowerright);

  \end{tikzpicture}
  \caption{Block diagram of feedback interconnection.}
  \label{fig:blockdiagram}
  \end{center}
\end{figure}

\subsection{Graph separation and multipliers for feedback stability}\label{sec:background_multipliers}

Let $A \in \mM_{m,n}$ and $B \in \mM_{n,m}$, and consider the (negative) feedback interconnection as shown in Figure~\ref{fig:blockdiagram}.
This interconnection is said to be stable if for each $(p,q) \in \mC^{n + m}$ there exists a unique vector $(r, w) \in \mC^{n + m}$.
From the representation in Figure~\ref{fig:blockdiagram}, it follows that $r = p + Bw$ and $w = q - Ar$, or equivalently that
\[
\begin{bmatrix}I_n & -B \\ A & I_m \end{bmatrix} \begin{bmatrix} r \\ w \end{bmatrix} = \begin{bmatrix} p \\ q \end{bmatrix}.
\]
Thus, the interconnection is stable if and only if the matrix $\begin{bsmallmatrix}I_n & -B \\ A & I_m \end{bsmallmatrix}$ is invertible, i.e, if and only if $\det( \begin{bsmallmatrix}I_n & -B \\ A & I_m \end{bsmallmatrix}) \neq 0$.
The latter is true if and only if $\det(I_m + AB) \neq 0$.

Next, let the graph of a matrix $C \in \mM_{m,n}$ be defined as all ordered pairs $(x_1,x_2) \in \mC^{n + m}$ such that $Cx_1 = x_2$, and the inverse graph be defined as all ordered pairs $(x_2, x_1)$. Then, $\det(I_m + AB) = 0$ if and only if there exists a nonzero $(x_1, x_2) \in \mC^{n+m}$ such that
\[
0 = \begin{bmatrix}I_n & -B \\ A & I_m \end{bmatrix} \begin{bmatrix}x_1 \\ x_2 \end{bmatrix} = \begin{bmatrix}x_1 - Bx_2 \\ Ax_1 + x_2 \end{bmatrix},
\]
where $(x_1, x_2)$ is identified as being a nontrivial element in both the graph of $-A$ and the inverse graph of $B$. Therefore, $\det(I_m + AB) \neq 0$ if and only if the graph of $-A$ and the inverse graph of $B$ only intersect in the origin, i.e., if and only if
\begin{equation}\label{eq:graph_sep}
\range \left( \begin{bmatrix}
I_{n} \\ -A
\end{bmatrix}
\right)
\cap
\range \left( \begin{bmatrix}
B \\ I_{m}
\end{bmatrix} \right) = \{ 0 \},
\end{equation}
where $\range(\cdot)$ denotes the column range of a matrix. A similar condition holds for the stability of a well-defined interconnection of dynamical systems, see, e.g., \cite{foias1993robust, doyle1993parallel}.

In \cite{iwasaki1998well}, it was shown that a necessary and sufficient condition for \eqref{eq:graph_sep} to hold is that there exists a multiplier which achieves quadratic separation.
More precisely,
\eqref{eq:graph_sep} holds 
if and only if there exists a $P \in \mH_{n+m}$ such that
\begin{subequations}\label{eq:multipliers}
\begin{align}
\begin{pmatrix}
I & -A^*
\end{pmatrix} P
\begin{pmatrix}
I \\ -A
\end{pmatrix} & \prec 0 \label{eq:multipliers_A}\\
\begin{pmatrix}
B^* & I
\end{pmatrix} P
\begin{pmatrix}
B \\ I
\end{pmatrix} & \succeq 0. \label{eq:multipliers_B}
\end{align}
\end{subequations}
There are several equivalent forms of this condition. For example, if there exists a $P \in \mH_{n+m}$ such that \eqref{eq:multipliers} holds, due to the strict inequality in \eqref{eq:multipliers_A} this $P$ also satisfies
\begin{subequations}\label{eq:multipliers4}
\begin{align}
\begin{pmatrix}
I & -A^*
\end{pmatrix} P
\begin{pmatrix}
I \\ -A
\end{pmatrix} & \preceq -\varepsilon A^*A \label{eq:multipliers4_A}\\
\begin{pmatrix}
B^* & I
\end{pmatrix} P
\begin{pmatrix}
B \\ I
\end{pmatrix} & \succeq 0, \label{eq:multipliers4_B}
\end{align}
\end{subequations}
for some $\varepsilon > 0$. Moreover, by rewriting the $(2,2)$-block of the block-matrix $P= \begin{bsmallmatrix} P_{11} & P_{12} \\ P_{12}^T & P_{22}\end{bsmallmatrix}$ as $P_{22} = \tilde{P}_{22} - \varepsilon I_m$ in \eqref{eq:multipliers4}, with a slight abuse of notation we see that there exists another matrix $P$ such that
\begin{subequations}\label{eq:multipliers2}
\begin{align}
\begin{pmatrix}
I & -A^*
\end{pmatrix} P
\begin{pmatrix}
I \\ -A
\end{pmatrix} & \preceq 0 \label{eq:multipliers2_A}\\
\begin{pmatrix}
B^* & I
\end{pmatrix} P
\begin{pmatrix}
B \\ I
\end{pmatrix} & \succ 0, \label{eq:multipliers2_B}
\end{align}
\end{subequations}
which in turn implies that
\begin{subequations}\label{eq:multipliers3}
\begin{align}
\begin{pmatrix}
I & -A^*
\end{pmatrix} P
\begin{pmatrix}
I \\ -A
\end{pmatrix} & \preceq 0 \label{eq:multipliers3_A}\\
\begin{pmatrix}
B^* & I
\end{pmatrix} P
\begin{pmatrix}
B \\ I
\end{pmatrix} & \succeq  \varepsilon B^*B, \label{eq:multipliers3_B}
\end{align}
\end{subequations}
for some $\varepsilon > 0$. Finally, by rewriting the $(1,1)$-block of $P$ in \eqref{eq:multipliers3} as $P_{11} = \tilde{P}_{11} + \varepsilon I_n$, we have that the existence of a multiplier fulfilling \eqref{eq:multipliers3} implies that there exists a multiplier fulfilling \eqref{eq:multipliers}.
This shows that the conditions \eqref{eq:multipliers}-\eqref{eq:multipliers3} are equivalent.
Nevertheless, in the step from \eqref{eq:multipliers4} to \eqref{eq:multipliers2}, and from \eqref{eq:multipliers3} to \eqref{eq:multipliers}, the actual multiplier (and hence  also potentially the structure) changes.
Since the results in this paper are concerned with necessary conditions for existence of multipliers of certain
structures, we state all these cases explicitly.
For convenience we summarize the results in the following lemmas.

\begin{lemma}
Let $A \in \mM_{m,n}$ and $B \in \mM_{n,m}$. If there exists a multiplier $P \in \mH_{n+m}$ such that any of the four conditions \eqref{eq:multipliers}-\eqref{eq:multipliers3} is satisfied, then there exists (potentially different) multipliers such that all the other three conditions are also satisfied.
\end{lemma}

\begin{lemma}[\cite{iwasaki1998well}] \label{lem: matrix_stability}
For $A \in \mM_{m,n}$ and $B \in \mM_{n,m}$, the following statements are equivalent:
\begin{enumerate}[(i)]
\item $\det(I_m + AB) \neq 0$;
\item condition \eqref{eq:graph_sep} holds;
\item there exists a matrix $P \in \mH_{n+m}$ such that \eqref{eq:multipliers} holds.
\end{enumerate}
\end{lemma}

\subsection{LTI systems and multipliers for feedback stability}\label{sec:background_LTI}

Next, we consider extensions of the aforementioned results to LTI systems. To this end, let us first introduce the function spaces needed (see, e.g., \cite{zhou1996robust} or \cite{vinnicombe2001uncertainty} for more details). Let $\|\cdot\|_2$ denote the matrix 2-norm, and let $\esssup$ denote the essential supremum of a function. Define the Lebesgue space
\begin{align*}
  \Linfty^{m \times n} & \!:=\! \left\{ \! \phi : j\mR \rightarrow \mM_{m,n} \left| \|\phi\|_\infty := \esssup_{\omega \in \mR} \|\phi(j\omega)\|_2 < \infty \!
    \right. \right\}
\end{align*}
and the Hardy space
\begin{align*}
  \Hinfty^{m \times n} & \!:=\! \left\{ \! \phi \in \Linfty^{m \times n} \left|
      \begin{array}{l} \phi \text{ has analytic continuation}\\
        \text{into } \mC_{+} \text{ with } \textstyle\esssup_{s \in \mC_{+}} \|\phi(s)\|_2 \\
        = \esssup_{\omega \in \mR} \|\phi(j\omega)\|_2 < \infty \end{array} \right. \!\!\!\!\right\} \! .
\end{align*}
The latter is the space of all stable transfer functions.
Denote by $\RR^{m \times n}$ the set of $m \times n$ real-rational proper transfer function matrices, and let $\RR\Hinfty^{m \times n} := \RR^{m \times n} \cap \Hinfty^{m \times n}$, i.e., the subset of $\RR^{m \times n}$ with no poles in the closed right-half complex plane. A $G \in \RR\Hinfty^{n \times n}$ is said to be passive if $G(j\omega) + G(j\omega)^* \succeq 0$ for all $\omega \in \mR$, and it is said to be output strictly passive if there exists $\epsilon > 0$ such that $G(j\omega) + G(j\omega)^* \succeq \epsilon G(j\omega)^* G(j\omega)$ for all $\omega \in \mR$.

Next, consider $G \in \RR\Hinfty^{m \times n}$ and $K \in \RR\Hinfty^{n \times m}$. Akin to the matrix setting, the (negative) feedback interconnection of $G$ and $K$ is said to be
stable if $(I + GK)^{-1} \in \RR\Hinfty^{m \times m}$. The following sufficient conditions for feedback stability are significantly important --- the first part of the result is well known whereas the second is less so.

\begin{proposition} \label{prop:IQC_suff}
Let $G \in \RR\Hinfty^{m \times n}$ and $K \in \RR\Hinfty^{n \times m}$. Then $(I + GK)^{-1} \in \RR\Hinfty^{m \times m}$ if there exists $\Pi
= [\begin{smallmatrix} \Pi_{11} & \Pi_{12} \\ \Pi_{21} & \Pi_{22} \end{smallmatrix}] \in
\Linfty^{(n + m) \times (n + m)}$ such that for all $\omega \in [0, \infty]$, $\Pi(j\omega)^* = \Pi(j\omega)$, $\Pi_{11}(j\omega) \preceq  0$,
$\Pi_{22}(j\omega) \succeq 0$,
\begin{align} \label{eq: IQC_suff}
  \begin{split}
\begin{pmatrix}
I & -G(j\omega)^*
\end{pmatrix} \Pi(j\omega)
\begin{pmatrix}
I \\ -G(j\omega)
\end{pmatrix} & \prec 0 \\
\begin{pmatrix}
K(j\omega)^* & I
\end{pmatrix} \Pi(j\omega)
\begin{pmatrix}
K(j\omega) \\ I
\end{pmatrix} & \succeq 0,
\end{split}
\end{align}
or equivalently,
\begin{align*}
\begin{pmatrix}
I & -G(j\omega)^*
\end{pmatrix} \Pi(j\omega)
\begin{pmatrix}
I \\ -G(j\omega)
\end{pmatrix} & \preceq 0 \\
\begin{pmatrix}
K(j\omega)^* & I
\end{pmatrix} \Pi(j\omega)
\begin{pmatrix}
K(j\omega) \\ I
\end{pmatrix} & \succeq \epsilon K(j\omega)^*K(j\omega)
\end{align*}
for some $\epsilon > 0$. Furthermore, if $m = n$ and $G^{-1}, K^{-1} \in \RR\Hinfty^{m \times m}$, then $(I + GK)^{-1} \in \RR\Hinfty^{m \times m}$ if there exists $\Pi \in
\Linfty^{(2m) \times (2m)}$ such that for all $\omega \in [0, \infty]$, $\Pi(j\omega)^* = \Pi(j\omega)$, $\Pi_{11}(j\omega) \succeq  0$,
$\Pi_{22}(j\omega) \preceq 0$, and \eqref{eq: IQC_suff} holds.
\end{proposition}

\begin{proof}
  If $\Pi_{11}(j\omega) \preceq 0$ and $\Pi_{22}(j\omega) \succeq 0$, then
  \[
\begin{pmatrix}
K(j\omega)^* & I
\end{pmatrix} \Pi(j\omega)
\begin{pmatrix}
K(j\omega) \\ I
\end{pmatrix} \succeq 0
\]
is equivalent to
  \[
\begin{pmatrix}
\homoparam K(j\omega)^* & I
\end{pmatrix} \Pi(j\omega)
\begin{pmatrix}
\homoparam K(j\omega) \\ I
\end{pmatrix} \succeq 0
\]
for all $\homoparam \in [0, 1]$. Feedback stability can then be established using the Parseval-Plancherel theorem as in \cite[Thm. 3.1]{megretski1993power}
and the theory of integral quadratic constraints~\cite[Thm. 1]{megretski1997system} or \cite[Cor. IV.3]{khong2021iqc}, where the proofs are written
purely in the time domain. An alternative, more direct frequency-domain proof is provided below for completeness.

By applying Lemma~\ref{lem: matrix_stability} frequency-wise, it holds that $\det(I + \homoparam G(j\omega)K(j\omega)) \neq 0$ for all
$\omega \in [0, \infty]$, $\homoparam \in [0, 1]$. It remains to show that $\det(I + \homoparam G(s)K(s)) \neq 0$ for all $s \in \mC_+$, from which
$(I + GK)^{-1} \in \RR\Hinfty^{m \times m}$ follows. To this end, observe that since $GK \in \RR\Hinfty^{m \times m}$,
$\det(I + \homoparam G(s)K(s)) \neq 0$ for all $s \in \mC_+$ for sufficiently small $\homoparam > 0$. Suppose to the contrapositive that $\det(I + G(s)K(s)) = 0$
for some $s \in \mC_+$. Then, by the continuity of the locations of the zeros of $\det(I + \homoparam G(s)K(s))$ in $\homoparam$, there must exist
an $\homoparam \in (0, 1)$ and an $\omega \in [0, \infty]$ such that $\det(I + \homoparam G(j\omega)K(j\omega)) = 0$, leading to a contradiction. Therefore, it must be
true that $\det(I + \homoparam G(s)K(s)) \neq 0$ for all $s \in \mC_+$ and $\homoparam \in [0, 1]$.

On the other hand, if $\Pi_{11}(j\omega) \succeq 0$ and $\Pi_{22}(j\omega) \preceq 0$, then
 \[
\begin{pmatrix}
K(j\omega)^* & I
\end{pmatrix} \Pi(j\omega)
\begin{pmatrix}
K(j\omega) \\ I
\end{pmatrix} \succeq 0
\]
is equivalent to
  \[
\begin{pmatrix}
\homoparam K(j\omega)^* & I
\end{pmatrix} \Pi(j\omega)
\begin{pmatrix}
\homoparam K(j\omega) \\ I
\end{pmatrix} \succeq 0
\]
for all $\homoparam \geq 1$. Since $G^{-1}, K^{-1} \in \RR\Hinfty^{m \times m}$, by the large gain theorem~\cite[Thm. 4.1]{zahedzadeh2008input},
$(I + \homoparam GK)^{-1} \in \RR\Hinfty^{m \times m}$ for sufficiently large $\homoparam \geq 1$. By repeating the preceding arguments, one may then establish
that $\det(I + \homoparam G(s)K(s)) \neq 0$ for all $s \in \mC_+$ and $\homoparam \geq 1$, from which $(I + GK)^{-1} \in \RR\Hinfty^{m \times m}$ follows.
\end{proof}

  \begin{remark} \label{rem:IQC_suff}
Proposition~\ref{prop:IQC_suff} remains true when all the inequality signs therein are flipped.
  \end{remark}
  
  The following necessary condition for feedback stability, complementing the sufficient condition in Proposition~\ref{prop:IQC_suff}, can be proved by using a construction from \cite{iwasaki1998well}.

\begin{proposition} \label{prop:IQC_nec}
Let $G \in \RR\Hinfty^{m \times n}$ and $K \in \RR\Hinfty^{n \times m}$. Then $(I + GK)^{-1} \in \RR\Hinfty^{m \times m}$ only if there exists $\Pi \in
\Linfty^{(n + m) \times (n + m)}$ such that for all $\omega \in [0, \infty]$, $\Pi(j\omega)^* = \Pi(j\omega)$, 
\begin{align*}
\begin{pmatrix}
I & -G(j\omega)^*
\end{pmatrix} \Pi(j\omega)
\begin{pmatrix}
I \\ -G(j\omega)
\end{pmatrix} & \prec 0 \\
\begin{pmatrix}
K(j\omega)^* & I
\end{pmatrix} \Pi(j\omega)
\begin{pmatrix}
K(j\omega) \\ I
\end{pmatrix} & \succeq 0. 
\end{align*}
\end{proposition}

\begin{proof}
That $(I + GK)^{-1} \in \RR\Hinfty^{m \times m}$ implies that $\inf_{\omega \in \mR}|\det(I + K(j\omega) G(j\omega))|^2 > 0$ for all $\omega \in
[0, \infty]$. Following the proof in~\cite[Cor. 1]{{iwasaki1998well}} frequency-wise, define
\[
\Pi(j\omega) := \begin{pmatrix}
G(j\omega)^* \\ I
\end{pmatrix}
\begin{pmatrix}
G(j\omega) & I
\end{pmatrix} - \epsilon I.
\]
The claim may then be verified to hold for sufficiently small $\epsilon > 0$.
\end{proof}

\subsection{Sectorial matrices and matrix phases}\label{sec:background_phase}
The numerical range, also called the field of values, of a matrix $A \in \mM_n$ is defined as
\[
W(A) := \big\{ z \in \mC \mid z = x^*Ax, \; x \in \mC^n, \; \| x \|^2 := x^*x = 1 \big\}.
\]
By the Toeplitz-Hausdorff theorem, for any $A \in \mM_n$ the numerical range $W(A)$ is a compact convex subset of $\mC$, see, e.g,. \cite[Property 1.2.1 and 1.2.2]{horn1994topics}, \cite[Thm.~4.1]{zhang2011matrix}, or \cite[Thm.~1.1-2]{gustafson1997numerical}.
Moreover, the numerical range of a matrix always contains its eigenvalues \cite[Property 1.2.6]{horn1994topics}.
Next, the conic hull of $W(A)$, i.e., the smallest convex cone that contains the numerical range, is given by the set
\[
W'(A) := \big\{ z \in \mC \mid z = x^*Ax, \; x \in \mC^n, \, x \neq 0 \big\},
\]
which is called the angular numerical range \cite[Def. 1.1.2]{horn1994topics}.
In particular, by the convexity of $W(A)$ it follows that if $0 \not \in W(A)$, then $ W(A)$ is contained in an open half-plane and hence the opening angle
of $W'(A)$ is strictly less than $\pi$ --- such matrices are called \emph{sectorial}.
If $0 \not \in \intr W(A)$, i.e., not in the interior, then $ W(A)$ is contained in a closed half-plane and hence the opening angle of $W'(A)$ is less than or equal to $\pi$ --- such matrices are called \emph{semi-sectorial}. If $0 \in \intr W(A)$, then $W'(A) = \mC$, and the opening angle is defined to be $2\pi$ (cf.~\cite[Def. 1.1.3]{horn1994topics}). Clearly, all sectorial matrices are also semi-sectorial. However, there exist
matrices that are not sectorial but for which the opening angle of the angular numerical range is strictly less than $\pi$; see Remark~\ref{rem:quasi-sectorial} for details.
In light of this, we define
the set of \emph{quasi-sectorial} matrices
as the set of all semi-sectorial matrix $A$ with opening angle of $W'(A)$ strictly less than $\pi$.
This definition gives the (strict) inclusions
\[
\text{sectorial } \subset \text{ quasi-sectorial } \subset \text{ semi-sectorial.}
\]
Finally, an important subset of sectorial matrices is the set of \emph{strictly accretive} matrices, which is defined as
\[
\mA_n := \{ A \in \mM_n \mid A + A^* \succ 0 \}.
\]
The closure of this set is the set of accretive matrices, i.e., the set of all matrices $ A \in \mM_n$ such that $A + A^* \succeq 0$, which is a subset of the semi-sectorial matrices.
In relation to this, we also define a matrix to be \emph{quasi-strictly accretive} if it is accretive and quasi-sectorial (cf.~Remark~\ref{rem:quasi-sectorial}).

All sectorial matrices can be diagonalized by congruence \cite{deprima1974range, furtado2001spectral, johnson2001generalization, zhang2015matrix}. More specifically, any sectorial matrix $A$ can be written as $A = T^*DT$, where $T \in \gln$ and where $D \in \mU_n$ is diagonal. This is called the \emph{sectorial factorization} \cite{zhang2015matrix}, and the matrix $D$ is unique up to ordering of the diagonal elements \cite{johnson2001generalization, zhang2015matrix}. Based on this factorization, following \cite{wang2020phases} we define the phases of a sectorial matrix to be the phases of the eigenvalues of $D$, and denote them by 
\[
\phi(A) = \begin{bmatrix}
\phi_1(A), & \phi_2(A), & \ldots, & \phi_n(A)
\end{bmatrix}^T.
\]
Each phase is only defined modulo $2\pi$, but by convention we sort them nonincreasingly, i.e., as
\[
\phimax(A) := \phi_1(A) \geq \phi_2(A) \geq \cdots \geq \phi_n(A) =: \phimin(A),
\]
and define them so that $\phimax(A) - \phimin(A) < \pi$. With this convention, we can for example see that that the phases of a sectorial matrix are invariant under congruence transformations, and that strictly accretive matrices are sectorial matrices with phases contained in $(-\pi/2, \pi/2)$ modulo $2\pi$.
The phases of a sectorial matrix has many nice properties, and can for example be used to guarantee that a matrix of the form $I + AB$ is of full rank; for an in-depth treatment of matrix phases we refer the reader to \cite{wang2020phases}. Moreover, the definition of phases can be extended to all semi-sectorial matrices; for the extension to quasi-sectorial matrices see Remark~\ref{rem:quasi-sectorial} below, and for the extension in the general case see \cite{furtado2003spectral, chen2024phase, wang2023phases}
for details. In any case, we still use $\phimax(A)$ and $\phimin(A)$ to denote the larges and smallest phase, respectively.

\begin{remark}\label{rem:quasi-sectorial}
Since the eigenvalues of a matrix are contained in its numerical range, any sectorial matrix must be full rank. The set of quasi-sectorial matrices extends the sectorial matrices to the set of matrices $A$ for which the opening angle of $W'(A)$ is strictly less than $\pi$, but that are not necessarily of full rank. In particular, let $A \in \mM_n$ be a quasi-sectorial but not sectorial matrix. Then the origin must be a sharp point on $\bdry W(A)$, i.e., the boundary of $W(A)$. This implies that $0$ is a normal eigenvalue of $A$, and that there exists a $U \in \mU_n$ such that
\[
A = U \begin{bmatrix}
0 & 0 \\
0 & \tilde{A}
\end{bmatrix} U^*
\]
where $\tilde{A}$ is sectorial and $\rank(\tilde{A}) = \rank(A)$
\cite[Thm.~1.6.6]{horn1994topics}.
The phases of a quasi-sectorial matrix is hence defined as the phases of $\tilde{A}$, and quasi-sectorial matrix thus have between $1$ and $n$ phases.
\end{remark}

\subsubsection*{The use of phases in MIMO LTI systems}
The concepts of magnitude and phase are well-established in the context of single-input-single-output LTI systems, and they both constitute highly useful and complementary
tools. However, while the concept of system gain has a generally accepted and useful generalization to MIMO LTI systems, including small-gain theorems for robust stability, the concept of phase has attracted much less attention. Early works trying to establish definitions of phases with useful properties in the MIMO setting can be found in, e.g.,
\cite{postlethwaite1981principal, owens1984numerical, anderson1988hilbert, haddad1992there, chen1998multivariable}. Recently, there has been a renewed
interest in the concept of phases for MIMO systems, both for LTI systems
\cite{chen2019phase, chen2024phase, mao2021phase}
and for nonlinear systems
\cite{chen2020phase}, with small-phase theorems for robust stability as a result. This concept of phases for MIMO LTI systems builds on the concept of matrix phases \cite{wang2020phases}, as introduced above, and can also be seen as a quantitive generalization of passive and negative imaginary systems \cite{chen2024phase}.
As will be seen below it is also connected to quadratic graph separation. In fact, this notion of phase turns out to be, in some sense, the correct notion in order to guarantee robust stability against certain types of magnitudinal uncertainties (see Section~\ref{sec:phase_multipliers}).

\section{Multipliers of phase type}\label{sec:phase_multipliers}
In this section we investigate the necessity of certain multipliers of phase type for robust stability of feedback interconnections with respect to magnitudinal uncertainties. In particular, we first show that $I + AB$ is nonsingular for magnitude scaling and certain congruence transformations, respectively, only if there exists certain types of phasal multipliers. The results are then extended to MIMO LTI systems.

\subsection{Multipliers for stability under scaling uncertainty}
One of the simplest forms of uncertainty is an uncertainty in the scaling of one of the matrices. In order to guarantee that the interconnection is stable for all scalings, it would therefore be desirable to show that $I + \tau AB$ is nonsingular for all $\tau \in \mR_{+}$. For $A,B \in \gln$, that is equivalent to that $\lambda(AB) \cap \mR_{-} = \emptyset$, i.e., that the intersection is empty, and necessary and sufficient conditions for the latter is given in the following proposition.

\begin{proposition}\label{prop:H_sect}
Given $A,B \in \gln$, there exists an $H \in \gln$ such that $HA$ and $H^*B$ are strictly accretive if and only if $\lambda(AB) \cap \mR_{-} = \emptyset$.
\end{proposition}

\begin{proof}
The proof follows by using results in \cite{ballantine1975accretive, wang2020phases}. More precisely, first assume $\lambda(AB) \cap \mR_- = \emptyset$. Then, by \cite[Thm.~1]{ballantine1975accretive} we have that the matrix $AB$ can be factored as $\tilde{A}\tilde{B}$, where $\tilde{A}, \tilde{B} \in \mA_n$. Let $H^* = \tilde{B}B^{-1}$, then $H^*B = \tilde{B} \in \mA_n$. Moreover, by congruence we have that $HA$ is accretive if and only if $H^{-1}(HA)H^{-*}=AH^{-*}$ is accretive. For the latter, we have that $AH^{-*} = AB\tilde{B}^{-1} = \tilde{A}\tilde{B}\tilde{B}^{-1} = \tilde{A} \in \mA_n$, and hence there exists an $H \in \gln$ so that $HA, H^*B \in \mA_n$. This proves the ``if''-statement.
To show the ``only if''-statement, assume that there exists an $H$ so that $HA, H^*B \in \mA_n$. Again, by congruence $HA \in \mA_n$ if and only if $AH^{-*} \in \mA_n$. By \cite[Thm.~6.2]{wang2020phases} it follows that $AB = AH^{-*}H^*B$ have no eigenvalues along $\mR_-$.
\end{proof}

The result in Proposition~\ref{prop:H_sect} can be understood in terms of the existence of a phasal multiplier $P \in \mH_{2n}$ that fulfills \eqref{eq:multipliers}, i.e., a multiplier $P$ where only the off-diagonal blocks are nonzero and where in fact both inequalities in \eqref{eq:multipliers} are strict (see \cite[Cor. VI.2]{meinsma1997dual}). 
In particular, this formally confirms the intuition that in order to show that the interconnection is stable under an arbitrary positive scaling uncertainty,
a certain type of ``phase information'' is the only thing that is needed.
Moreover, these results can be strengthened to (certain) matrices which are not of full rank as follows.

\begin{theorem}\label{thm:H_sectorial}
Given $A, B \in \mM_n$, assume that if zero is an eigenvalue of $AB$, then it is semi-simple.%
\footnote{An eigenvalue is called semi-simple if its algebraic and geometric multiplicities are the same \cite[Def.~1.4.3]{horn2013matrix}.
This is equivalent with that all Jordan blocks corresponding to the eigenvalue are of size $1 \times 1$ \cite[Prob.~3.1.P5]{horn2013matrix}. 
}
Then the following statements are equivalent:
\begin{enumerate}[(i)]
\item $\det(I + \tau AB) \neq 0$ for all $\tau \geq 0$;
\item there exists a $P \in \mH_{2n}$ fulfilling \eqref{eq:multipliers4}, and $P$ takes the form
\begin{equation}\label{eq:phase_multiplier}
P = \begin{bmatrix}
0 & H \\ H^* & 0
\end{bmatrix}
\end{equation}
for some $H \in \mM_n$;
\item for the eigenvalues of $AB$, it holds that
\begin{equation}\label{eq:condition}
\lambda(AB) \cap \mR_{--} = \emptyset.
\end{equation}
\end{enumerate}
\end{theorem}

\begin{proof}
See Appendix~\ref{app:proof_thm:H_sectorial}.
\end{proof}

If the matrix A in Theorem~\ref{thm:H_sectorial} is full rank, then the statement in Theorem~\ref{thm:H_sectorial}(ii) can be strengthened and a number of other equivalent conditions can also be derived. In particular, the multiplier $H$ can be chosen to be nonsingular and strict accretiveness of $HA$ can be guaranteed.

\begin{corollary}\label{cor:H_sectorial_A_inv}
Let  $A, B \in \mM_n$ be as in Theorem~\ref{thm:H_sectorial}. If $A \in \gln$, then the statements in Theorem~\ref{thm:H_sectorial} are also equivalent to
\begin{itemize}
\item[(iv)] there exists an $H \in \gln$ such that $HA$ is strictly accretive and $H^*B$ is quasi-strictly accretive;
\item[(v)] there exists an $H \in \gln$ such that $HA$ is strictly accretive and $H^*B$ is accretive.
\end{itemize}
Moreover, the multiplier $P$ in Theorem~\ref{thm:H_sectorial}(ii) can be selected so that it fulfills \eqref{eq:multipliers}.
\end{corollary}

\begin{proof}
See Appendix~\ref{app:proof_thm:H_sectorial}.
\end{proof}

In many applications, we would be interested in corresponding results for real-valued matrices. By just slightly modifying the proof of the theorem, we have the following corollary.

\begin{corollary} \label{cor: H_sectorial}
Under the assumptions in Theorem~\ref{thm:H_sectorial}, if $A, B \in \mM_n(\mR)$, the same conclusion is true where we can restrict $H$ to also be real.
\end{corollary}

\begin{proof}
See Appendix~\ref{app:proof_thm:H_sectorial}.
\end{proof}

Observe that, in general, it is not possible to relax the assumption in Theorem~\ref{thm:H_sectorial} that if zero is an eigenvalue of $AB$, then it is semi-simple.
This can be seen by the following counterexample for $3 \times 3$ matrices, where the zero-eigenvalue of $AB$ has a Jordan block of size $2 \times 2$.

\begin{example}
Let
\[
A = \begin{bmatrix}
1 & 0 & 0 \\
0 & 0 & 1 \\
0 & 0 & 0
\end{bmatrix}, \quad
B = I_3,
\quad
H = \begin{bmatrix}
h_{11} & h_{12} & h_{13} \\
h_{21} & h_{22} & h_{23} \\
h_{31} & h_{32} & h_{33}
\end{bmatrix},
\]
and note that $\det(I + \tau AB) = 1 + \tau \neq 0$ for all $\tau \geq 0$. Moreover, $A^*A = \diag(1, 0, 1)$.
Next, note that the existence of a multiplier of the form \eqref{eq:phase_multiplier} that fulfills any of the conditions \eqref{eq:multipliers}-\eqref{eq:multipliers3} would imply that both $HA + A^*H^* \succeq 0$ and $H^*B + B^*H \succeq 0$.
A direct calculation gives that
\[
HA + A^*H^* = \begin{bmatrix}
h_{11} + h_{11}^* & h_{12} + h_{31}^* & h_{13} \\
h_{31} + h_{12}^* & h_{32} + h_{32}^* & h_{33} \\
h_{13}^*          & h_{33}^*          & 0 
\end{bmatrix},
\]
and for this to be positive semidefinite we must have $h_{13} = h_{33} = 0$, see, e.g., \cite[Obs.~7.1.10]{horn2013matrix}. Therefore, $HA + A^*H^*$ has at most rank 2, and can hence only be positive semidefinite. This means that there is no multiplier of the form \eqref{eq:phase_multiplier} that fulfills \eqref{eq:multipliers}. Moreover, it is easily seen that $HA + A^*H^* \not \succeq \varepsilon A^*A$ for all $\epsilon > 0$, and therefore there is no multiplier of the form \eqref{eq:phase_multiplier} that satisfies \eqref{eq:multipliers4}.
Next, note that
\[
H^*\!B + B^*H \! \! = \! H + H^* \! \! = \! \!
\begin{bmatrix}
h_{11} + h_{11}^* & h_{12} + h_{21}^* & h_{31}^* \\
h_{21} + h_{12}^* & h_{22} + h_{22}^* & h_{23} + h_{32}^* \\
h_{31} & h_{32} + h_{23}^* & 0
\end{bmatrix}\!\!,
\]
which, similar to above, can only be positive semidefinite if $h_{31} = 0$ and $h_{32} = -h_{23}^*$. However, that means that $H + H^*$ has rank 2, and hence can only be positive semidefinite. Moreover, for all $\epsilon > 0$ we therefore also have that $H^*B + B^*H \not \succeq \epsilon B^*B$. Thus, there is no multiplier of the form \eqref{eq:phase_multiplier} that satisfies \eqref{eq:multipliers2} or \eqref{eq:multipliers3}.
\end{example}

Nevertheless, while the above counterexample shows that the condition on the semi-simple zero-eigenvalue can in general not be relaxed, the case for matrices of size $2 \times 2$ is still open.
The following gives an example of where there exists a multiplier of the form \eqref{eq:phase_multiplier} that fulfills \eqref{eq:multipliers4}, despite the fact that that zero is a not semi-simple eigenvalue of $AB$. 

\begin{example}
Let
\[
A =
\begin{bmatrix}
0 & 1 \\
0 & 0
\end{bmatrix},
\qquad
B = I_2,
\qquad
H = \begin{bmatrix}
0  & -1 \\
1 & 1 
\end{bmatrix},
\]
and note that $\det(I + \tau AB) = 1 \neq 0$ for all $\tau \geq 0$.
A direct calculation gives that $A^*A = \diag(0,1)$, that $HA = \diag(0,1)$, and that $H^*B + HB^* = H^* + H = \diag(0,2)$. Therefore, for $\epsilon = 1$ we have that $P$ as in \eqref{eq:phase_multiplier} fulfills \eqref{eq:multipliers4}.
\end{example}

\begin{remark}\label{rem:relation_to_mu}
Theorem~\ref{thm:H_sectorial} appears to be intrinsically and closely related to $\mu$-analysis when $AB$ is invertible. In particular, it may be established using the $\mu$-analysis results in~\cite{fan1991robustness, chou1999stability} that (i) in Theorem~\ref{thm:H_sectorial} implies there exists $H \in \gln$ such that $HAB + B^*A^*H^* > 0$ and $H + H^* > 0$, which may also be established via (iv) in Corollary~\ref{cor:H_sectorial_A_inv}. Further investigation into the delicate relation between Theorem~\ref{thm:H_sectorial} and $\mu$-analysis does not appear to be straightforward and is a worthwhile future research direction of significant importance.
\end{remark}

\subsection{Multipliers for stability under congruence}

The results in Theorem~\ref{thm:H_sectorial} show that stability under magnitude scaling is equivalent to the existence of a phasal multiplier. 
Interesting to note in this context is that for this limited (and in some sense minimal) set of magnitudinal perturbations, the set of possible multipliers of phase type to which we could restrict our attention, and still have a necessary and sufficient condition for robust stability, is large (and in some sense maximal). Motivated by this, we next investigate a type of perturbations against which a minimal set of phasal multipliers can guarantee robust stability. In this case, we have the following result.

\begin{theorem}\label{thm:TATstarSBSstar}
Given $A,B \in \mM_n \setminus \{ 0 \}$, the following statements are equivalent:
\begin{enumerate}[(i)]
\item $\det(I + T^*AT S^*BS) \neq 0$ for all $T,S \in \gln$;
\item there exists a $P \in \mH_{2n}$ fulfilling \eqref{eq:multipliers3} or \eqref{eq:multipliers4}, and which takes the form
\[
P = \begin{bmatrix}
0 & zI \\ z^*I & 0
\end{bmatrix}
\]
for some $z \in \mT$;
\item one matrix is quasi-sectorial, the other is semi-sectorial, $\phimax(A) + \phimax(B) < \pi$, and $\phimin(A) + \phimin(B) > -\pi$.
\end{enumerate}
Finally, if the quasi-sectorial matrix in (iii) is of full rank, then the multiplier $P$ in (ii) fulfills \eqref{eq:multipliers2} or \eqref{eq:multipliers}.
\end{theorem}

\begin{proof}
See Appendix~\ref{app:proof_thm:TATstarSBSstar}.
\end{proof}

\begin{remark}
The result in Theorem~\ref{thm:TATstarSBSstar} is a type of small-phase theorem, akin to \cite[Thm.~7.1]{wang2023phases}.
The difference is that Theorem~\ref{thm:TATstarSBSstar} considers robust stability against congruence of two given matrices, while \cite[Thm.~7.1]{wang2023phases} considers robust stability with respect to a matrix cone of semi-sectorial matrices.
Nevertheless, note that when $B = I$, the result in Theorem~\ref{thm:TATstarSBSstar} specializes to robust stability against the matrix cone $\mP_n$.
\end{remark}

Similar to before, we get the following real-valued version of the theorem as a corollary.

\begin{corollary} \label{cor: TATstarSBSstar}
Theorem~\ref{thm:TATstarSBSstar} remains true when $A$, $B$, $T$, $S$, and $z$ are all real.
\end{corollary}

\begin{proof}
This can be established by noting that a real matrix $A$ is semi-sectorial if and only if either $A + A^T \succeq 0$ or $A + A^T \preceq 0$.
\end{proof}

\subsection{Phasal multipliers for LTI systems}
Next, we extend the above results to LTI systems.
In particular, in Section~\ref{sec:background} it was shown how quadratic graph-separation results for matrices can be extended to LTI systems.
Here, we follow along the same line. In particular, for magnitudinal perturbations we have the following necessary and sufficient condition for stability.

\begin{theorem}\label{thm:LTI_magnitudinal_uncertainty}
  Given $G \in \RR\Hinfty^{n \times n}$ and $K \in \RR\Hinfty^{n \times n}$ for which any potential zero-eigenvalue of $G(j\omega)K(j\omega)$, for $\omega
  \in [0, \infty]$, is
  semi-simple, then $(I + \tau GK)^{-1} \in \RR\Hinfty^{n \times n}$ for all $\tau > 0$ if and only if there exists an $H \in \Linfty^{n \times n}$ such that
  for all $\omega \in [0, \infty]$,
\begin{align*}
\begin{pmatrix}
I & -G(j\omega)^*
\end{pmatrix} \Pi(j\omega)
\begin{pmatrix}
I \\ -G(j\omega)
\end{pmatrix} & \preceq -\epsilon G(j\omega)^*G(j\omega)  \\
\begin{pmatrix}
K(j\omega)^* & I
\end{pmatrix} \Pi(j\omega)
\begin{pmatrix}
K(j\omega) \\ I
\end{pmatrix} & \succeq 0,
\end{align*}
where
\[
  \Pi(j\omega) := \begin{bmatrix}
    0 & H(j\omega) \\
    H(j\omega)^* & 0
    \end{bmatrix}.
\]
\end{theorem}

\begin{proof}
  Sufficiency follows from Proposition~\ref{prop:IQC_suff}, Remark~\ref{rem:IQC_suff}, and the fact that the inequality
    \[
\begin{pmatrix}
K(j\omega)^* & I
\end{pmatrix} \Pi(j\omega)
\begin{pmatrix}
K(j\omega) \\ I
\end{pmatrix} \succeq 0
\]
implies that
  \[
\begin{pmatrix}
\tau K(j\omega)^* & I
\end{pmatrix} \Pi(j\omega)
\begin{pmatrix}
\tau K(j\omega) \\ I
\end{pmatrix} \succeq 0
\]
for all $\tau > 0$. Necessity can be established by applying Theorem~\ref{thm:H_sectorial} and Corollary~\ref{cor: H_sectorial} frequency-wise in a similar fashion to the proof of Proposition~\ref{prop:IQC_nec}. In particular, since $G$ and $K$ are continuous on the imaginary axis, $\Pi$ may also be chosen to be continuous on the imaginary axis.
\end{proof}

\begin{remark}
    By examining a transfer function matrix frequency-wise, analogous observations to those in Remark~\ref{rem:relation_to_mu} are applicable in the context of Theorem~\ref{thm:LTI_magnitudinal_uncertainty}.
\end{remark}

The result above shows that if the feedback interconnection is robustly stable against arbitrary positive scaling, then only phasal properties of the
open-loop components are required to establish its stability, i.e., any corresponding multiplier $\Pi$ has its diagonal blocks being 0.

Analogously, the following two results, which establish sufficiency and necessity for stability under real congruence transformations, may be readily derived.

\begin{theorem}\label{thm:LTI_TATstarSBSstar_suff}
Given $G \in \RR\Hinfty^{n \times n}$ and $K \in \RR\Hinfty^{n \times n}$, then $(I + T^TGT S^TKS)^{-1} \in \RR\Hinfty^{n \times n}$ for all $T,
S \in \gln(\mR)$ if there exists $z \in \Linfty$ such that for all $\omega \in [0, \infty]$,
\begin{align} \label{eq: IQC_sep1}
  \begin{split}
\begin{pmatrix}
I & -G(j\omega)^*
\end{pmatrix} \Pi(j\omega)
\begin{pmatrix}
I \\ -G(j\omega)
\end{pmatrix} & \preceq -\epsilon G(j\omega)^*G(j\omega)  \\
\begin{pmatrix}
K(j\omega)^* & I
\end{pmatrix} \Pi(j\omega)
\begin{pmatrix}
K(j\omega) \\ I
\end{pmatrix} & \succeq 0
\end{split}
\end{align}
or
\begin{align} \label{eq: IQC_sep2}
  \begin{split}
\begin{pmatrix}
I & -G(j\omega)^*
\end{pmatrix} \Pi(j\omega)
\begin{pmatrix}
I \\ -G(j\omega)
\end{pmatrix} & \preceq 0  \\
\begin{pmatrix}
K(j\omega)^* & I
\end{pmatrix} \Pi(j\omega)
\begin{pmatrix}
K(j\omega) \\ I
\end{pmatrix} & \succeq \epsilon K(j\omega)^* K(j\omega),
\end{split}
\end{align}
where
\[
\Pi(j\omega) := \begin{bmatrix}
0 & z(j\omega) I \\
z(j\omega)^* I & 0
\end{bmatrix}.
\]
\end{theorem}

\begin{proof}
  Observe that \eqref{eq: IQC_sep1} implies that for all $T, S \in \gln(\mR)$, there exists $\epsilon_T$ such that
  \begin{align*}
  \begin{split}
\begin{pmatrix}
I & -T^TG(j\omega)^* T
\end{pmatrix} \Pi(j\omega)
\begin{pmatrix}
I \\ -T^T G(j\omega) T
\end{pmatrix} & \preceq \\
-\epsilon_T T^TG(j\omega)^* T T^TG(j\omega)T  \\
\begin{pmatrix}
S^TK(j\omega)^*S & I
\end{pmatrix} \Pi(j\omega)
\begin{pmatrix}
S^TK(j\omega)S \\ I
\end{pmatrix} & \succeq 0,
\end{split}
  \end{align*}
  and similarly for \eqref{eq: IQC_sep2}. The claim then follows from Proposition~\ref{prop:IQC_suff} and Remark~\ref{rem:IQC_suff}.
\end{proof}

\begin{theorem}\label{thm:LTI_TATstarSBSstar_nec}
Given $G \in \RR\Hinfty^{n \times n}$ and $K \in \RR\Hinfty^{n \times n}$, $(I + T^TGT S^TKS)^{-1} \in \RR\Hinfty^{n \times n}$ for all $T,
S \in \gln(\mR)$ only if for $\omega \in \{0, \infty\}$,
\begin{align*}
  \begin{split}
\begin{pmatrix}
I & -G(j\omega)^*
\end{pmatrix} \Pi(j\omega)
\begin{pmatrix}
I \\ -G(j\omega)
\end{pmatrix} & \preceq -\epsilon G(j\omega)^*G(j\omega)  \\
\begin{pmatrix}
K(j\omega)^* & I
\end{pmatrix} \Pi(j\omega)
\begin{pmatrix}
K(j\omega) \\ I
\end{pmatrix} & \succeq 0
\end{split}
\end{align*}
or
\begin{align*} 
  \begin{split}
\begin{pmatrix}
I & -G(j\omega)^*
\end{pmatrix} \Pi(j\omega)
\begin{pmatrix}
I \\ -G(j\omega)
\end{pmatrix} & \preceq 0  \\
\begin{pmatrix}
K(j\omega)^* & I
\end{pmatrix} \Pi(j\omega)
\begin{pmatrix}
K(j\omega) \\ I
\end{pmatrix} & \succeq \epsilon K(j\omega)^* K(j\omega),
\end{split}
\end{align*}
where for each $\omega \in \{0, \infty\}$, 
\[
\Pi(j\omega) = \begin{bmatrix}
0 &  I \\
I & 0
\end{bmatrix} \text{ or } -\begin{bmatrix}
0 &  I \\
I & 0
\end{bmatrix}.
\]
\end{theorem}

\begin{proof}
This follows by applying Corollary~\ref{cor: TATstarSBSstar} to the pairs of real matrices $\{G(j0), K(j0)\}$ and $\{G(j\infty), K(j\infty)\}$.
\end{proof}

The separation condition in the theorem above holds for sufficiently small and large frequencies by the continuity of the transfer functions $G$ and
$K$. Such properties are useful, for instance, in the study of negative imaginary systems~\cite{khong2018robust}, where an example of open-loop
systems being passive on sufficiently small and large frequencies and negative imaginary elsewhere can be found.

\begin{example}
  Let $G \in \RR\Hinfty^{n \times n}$ be output strictly passive and $K \in \RR\Hinfty^{n \times n}$ be passive. Then they satisfy the separation
  conditions in all three of the theorems above with
\[
\Pi := \begin{bmatrix}
0 &  I \\
I & 0
\end{bmatrix}.
\]
This is a well-known passivity theorem.
\end{example}

\section{Multipliers of gain type}\label{sec:gain_multipliers}
In the previous section, we investigated the necessity of phasal multipliers in order to guarantee robust stability with respect to certain magnitudinal perturbations. In this section, we turn to the necessity of magnitudinal multipliers in order to guarantee robust stability with respect to certain phasal perturbations.

\subsection{Multiplier for stability under scalar rotation uncertainty}
In analogy with Section~\ref{sec:phase_multipliers}, we first consider scalar rotational uncertainties. In this case, we have the following result.

\begin{theorem}\label{thm:multiplier_phase_uncertainty}
Given $A \in \mM_{m,n}$ and $B \in \mM_{n,m}$, the following statements are equivalent:
\begin{enumerate}[(i)]
\item $\det(I + e^{j \theta} AB) \neq 0$ for all $\theta \in [0, 2\pi)$;
\item there exists a $P \in \mH_{n+m}$ fulfilling \eqref{eq:multipliers}, with both inequalities strict, which takes the form
\begin{equation}\label{eq:gain_multiplier}
P = \begin{bmatrix}
-N & 0 \\ 0 & M
\end{bmatrix}
\end{equation}
for some $N \in \mH_n$ and $M \in \mH_m$;
\item for the eigenvalues of $AB$, it holds that
\begin{equation}\label{eq:condition_2}
\lambda(AB) \cap \mT = \emptyset;
\end{equation}
\item  there exists $M \in \mH_m$ and $N \in \mH_n$ such $A^*MA \prec N$ and $B^*NB \prec M$.
\end{enumerate}
\end{theorem}

\begin{proof}
The equivalences ``(i) $\Leftrightarrow$ (iii)'' and ``(ii) $\Leftrightarrow$ (iv)'' are straightforward. We therefore restrict our attention to the equivalence ``(iii) $\Leftrightarrow$ (iv)''.
To this end, first note that the statement is trivial if any of the two matrices $A,B$ is the zero matrix. Therefore, in the remaining we will, without loss of generality, assume that both are nonzero.

To show ``(iv) $\Rightarrow$ (iii)'': assume that there exist $M \in \mH_m$ and $N \in \mH_n$ such that $A^*MA \prec N$ and $B^*NB \prec M$.
Together with \cite[Obs.~7.1.8]{horn2013matrix}, 
the former  inequality implies that
\[
B^*A^*MAB \preceq B^*NB,
\]
and hence
\[
B^*A^*MAB \preceq B^*NB \prec M.
\]
Let $Q := M - B^*A^*MAB \succ 0$. This means that $M$ is a solution to the Stein equation
\[
M - B^*A^*MAB = Q,
\]
where $Q \in \mP_m$, and hence by \cite[Thm.~13.2.2]{lancaster1985thetheory} we therefore have that $\lambda(AB) \cap \mT = \emptyset$.

To show the ``(iii) $\Rightarrow$ (iv)'': assume that \eqref{eq:condition_2} holds, and let $AB = XJX^{-1}$ be a Jordan decomposition of $AB$. By \eqref{eq:condition_2} we can, without loss of generality, assume that 
\[
J = \begin{pmatrix}
J_1 & 0\\
0 & J_2
\end{pmatrix},
\]
where $J_1 \in \mM_{m_1}$ has all eigenvalues in $\mD$ and $J_2 \in \mM_{m_2}$ has all eigenvalues in $(\bar{\mD})^C$, i.e., outside of the close unit disc, and where $m = m_1 + m_2$.
Next, using \cite[Sec.~13.2]{lancaster1985thetheory} and \cite[Exer.~4.9.30]{gohberg2006indefinite} we have that for any $P_1 \in \mP_{m_1}$ and $P_2 \in \mP_{m_2}$, there is at least one solution $M$ to the Stein equation%
\footnote{Solutions to the Stein equation can be obtained by using solutions to the discrete-time Lyapunov equation: one for the stable part, and one for the anti-stable part. In particular, in this case one solution is given by $M =
X^{-*} \diag(H_1, H_2) X^{-1}
$, where $H_1 = \sum_{\ell = 0}^\infty (J_1^*)^\ell P_1 J_1^\ell$ and $H_2 = -J_2^{-*}\left(\sum_{\ell = 0}^\infty (J_2^{-*})^\ell P_2 (J_2^{-1})^\ell \right) J_2^{-1} = -\sum_{\ell = 1}^\infty (J_2^{-*})^\ell P_2 (J_2^{-1})^\ell$. It can be verified that both $H_1$ and $H_2$ are well-defined, since $J_1$ has all eigenvalues in the open unit disc and $J_2$ has all eigenvalues outside of the closed unit disc.}
\begin{equation}\label{eq:stein_1}
M - B^*A^*MAB =
X^{-*}
\begin{pmatrix}
P_1 & 0 \\ 0 & P_2
\end{pmatrix}
X^{-1} := Q \succ 0.
\end{equation}
Let $M$ be a solution to \eqref{eq:stein_1}, in which case $B^*A^*MAB \prec M$.
Now, define $\tilde{N} := A^*MA$ and note that this implies that
\begin{align*}
& A^*MA = \tilde{N} \preceq \tilde{N}, \\
& B^*\tilde{N}B = B^*A^*MAB \prec M.
\end{align*}
To prove that there exist $M \in \mH_m$ and $N \in \mH_n$ with both inequalities above strict,
consider $N := \tilde{N} + \epsilon I$ for some $\epsilon > 0$. In particular,
\[
A^*MA = \tilde{N} \prec \tilde{N} + \epsilon I = N
\]
for all $\epsilon > 0$. Moreover, since $M - B^*\tilde{N}B = M - B^*AMAB  = Q \succ 0$,
we have that
\[
M - B^*NB = M - B^*\tilde{N}B - \epsilon B^*B = Q - \epsilon B^*B \succ 0,
\]
for $\epsilon$ small enough.
This  completes the proof.
\end{proof}

\subsection{Multiplier for stability under unitary perturbation}

In the previous section, we established that the existence of a magnitudinal multiplier is necessary and sufficient for stability in the case of a scalar rotational uncertainty. Similarly to Section~\ref{sec:phase_multipliers}, a minimal set of uncertainties gives rise to a maximal set of multipliers. Motivated by this, we now consider for which type of phasal uncertainties a much smaller (in some sense minimal) set of magnitudinal multipliers can guarantee robust stability. More specifically, the set of magnitudinal multipliers considered are diagonal and completely parametrized by a nonnegative number and an element that is either $1$ or $-1$, i.e., a nonzero element whose useful information is only its sign.

We start by establishing a lemma.
To state the result, recall the convention we use that a matrix $A \in \mM_{m,n}$ has $n$ singular values, which are given by $\sigma(A) = \sqrt{\lambda(A^*A)}$, and hence if $n > m$, then $\sigma_{m+1}(A) = \cdots = \sigma_n(A) = 0$.

\begin{lemma}\label{lem:multiplier_and_singular_values}
Given $A \in \mM_{m,n}$ and $B \in \mM_{n,m}$,
\begin{enumerate}[(i)]
\item there exists a $P \in \mH_{n+m}$ of the form $P = \diag(- \gamma^2I, I)$, $\gamma \in \mR$, fulfilling \eqref{eq:multipliers}, with both inequalities strict, if and only if $\sigma_1(A) \sigma_1(B) < 1$;
\item there exists a $P \in \mH_{n+m}$ of the form $P = \diag(\gamma^2I, -I)$, $\gamma \in \mR$, fulfilling \eqref{eq:multipliers}, with both inequalities strict, if and only if $\sigma_{n}(A) \sigma_{m}(B) > 1$.
\end{enumerate}
\end{lemma}

Before we proceed, note that the conditions in Lemma~\ref{lem:multiplier_and_singular_values}(ii) can only ever be fulfilled if $n = m$ and both matrices are full rank, since otherwise at least one of the two singular values $\sigma_{n}(A), \sigma_{m}(B)$ equals zero.

\begin{proof}
We start with proving (i). To this end, note that a direct calculation in \eqref{eq:multipliers} (with both inequalities strict) gives that a multiplier of the prescribed form exists if and only if 
\begin{equation}\label{eq:UAVB_proof}
A^*A \prec \gamma^2 I \qquad \text{and} \qquad
\gamma^2 B^*B \prec I,
\end{equation}
which is the case if and only if there exists a $\gamma \in \mR$ such that all singular values of $A$ are strictly smaller than $|\gamma|$, and all singular values of $B$ are strictly smaller than or $1/|\gamma|$. Therefore, the existence of such a multiplier clearly implies that $\sigma_1(A) \sigma_1(B) < 1$. Conversely, if $\sigma_1(A) \sigma_1(B) < 1$, then a direct calculation shows that $1/\sigma_1^2(A) - \sigma_1^2(B) > 0$ and that for any $0 < \epsilon < 1/\sigma_1^2(A) - \sigma_1^2(B)$, if we take $\gamma^2 = 1/(\sigma_1^2(B) + \epsilon)$ we have that $\gamma^2 > 1/(\sigma_1^2(B) + 1/\sigma_1^2(A) - \sigma_1^2(B)) = \sigma_1^2(A)$ and that $1/\gamma^2 > \sigma_1^2(B)$, and hence such $\gamma^2$ fulfills \eqref{eq:UAVB_proof}.

Next, to prove (ii) we follow along the same lines. However, first note that $\sigma_{n}(A) \sigma_{m}(B) > 1$ only if $n = m$ and both $A$ and $B$ are invertible, since otherwise at least one of the two singular values equals zero. Now, a multiplier of the prescribed form exists if and only if
\begin{equation}\label{eq:UAVB_proof_2}
\begin{aligned}
A^*A \succ \gamma^2 I \qquad  \text{and} \qquad \gamma^2 B^*B \succ I,
\end{aligned}
\end{equation}
which, similarly, can only hold if $n = m$ and both $A$ and $B$ are invertible. Henceforth, we can therefore restrict our attention to that case.
Now, \eqref{eq:UAVB_proof_2} holds if and only if there exists a $\gamma \in \mR$ such that all singular values of $A$ are strictly larger than $|\gamma|$, and all singular values of $B$ are strictly larger than $1/|\gamma|$. Therefore, the existence of such a multiplier implies that $\sigma_{n}(A) \sigma_{n}(B) > 1$. The converse statement is proved analogously, but by instead considering $0 < \epsilon < \sigma_n^2(B) - 1/\sigma_n^2(A)$ and $\gamma^2 = 1/(\sigma_{n}^2(B) - \varepsilon)$, which means that $\gamma^2 < 1/(\sigma_{n}^2(B) - \sigma_{n}^2(B) + 1/\sigma_{n}^2(A)) = \sigma_{n}^2(A)$ and that $1/\gamma^2 < \sigma_{n}^2(B)$. This proves the lemma.
\end{proof}

The preceding lemma considers two different domains in which stability of $I + AB$ can be guaranteed: when both $A$ and $B$ have either small gain or large gain.
In both cases, we expect that stability should be preserved under a suitable notion of rotation. We can now formalize this as follows.

\begin{theorem}\label{thm: UAVB}
Given $A \in \mM_{m,n}$ and $B \in \mM_{n,m}$, the following statements are equivalent:
\begin{enumerate}[(i)]
\item $\det(I + UAVB) \neq 0$ for all $U \in \mU_m$ and all $V \in \mU_n$;
\item there exists a $P \in \mH_{n+m}$ fulfilling \eqref{eq:multipliers}, with both inequalities strict, which takes the form
\begin{equation}\label{eq:P_gain}
P = \begin{bmatrix}
-\xi \gamma^2 I & 0 \\ 0 & \xi I
\end{bmatrix}
\end{equation}
for some $\gamma \in \mR$ and $\xi \in \{-1, 1\}$;
\item either $\sigma_{1}(A) \sigma_{1}(B) < 1$ or $\sigma_{n}(A) \sigma_{m}(B) > 1$.
\end{enumerate}
\end{theorem}

\begin{proof}
See Appendix~\ref{app:proof_thm: UAVB}.
\end{proof}

\begin{remark}
When $n = m$, both $A$ and $B$ in Theorem~\ref{thm: UAVB} have polar decompositions $A = U_A Q_A$ and $B = U_B Q_B$, where $U_A, U_B \in \mU_n$ and $Q_A, Q_B \in \mP_n$. In this case, the uncertainty can be interpreted as the principle phases of the two matrices being unknown; the principle phases are defined in \cite{postlethwaite1981principal}.
Moreover, if $B = I$, the result in Theorem~\ref{thm: UAVB} establishes conditions for robust stability against all elements in $\mU_n$, and in this case the equivalence between (i) and (iii) follows from \cite{horn1954eigenvalues}.
\end{remark}

\begin{remark}
Note that in a numerical implementation searching for multipliers to guarantee stability, the conditions in Theorem~\ref{thm: UAVB}(ii) can be relaxed to searching for multipliers of the form $P = \diag(-\eta_1 I, \eta_2 I)$, for $\eta_1, \eta_2 \in \mR$. This means that the search for multipliers fulfilling \eqref{eq:multipliers} can either be formulated as two LMIs of the form \eqref{eq:P_gain}, each of which has one unknown $\gamma^2 \geq 0$, or it can be solved as one LMI in the two unknowns $\eta_1, \eta_2 \in \mR$.
\end{remark}

\subsection{Magnitudinal multipliers for LTI systems}

In order to extend the above results to LTI systems, we first need the following definitions and results: a transfer function $U \in \RR\Hinfty^{n \times n}$ is said to be unitary if $U (j\omega) \in \mU_n$ for all $\omega \in [0, \infty]$. Moreover, by the proof of
\cite[Lem. 1.14]{vinnicombe2001uncertainty}, it holds that for every $\omega > 0$ and $X \in \mU_n$, there exists unitary $U \in \RR\Hinfty^{n
  \times n}$ such that $Q(j\omega) = X$.
Next, the following expression will be used in the forthcoming theorems:
\begin{align} \label{eq: gain_sep}
  \begin{split}
\begin{pmatrix}
I & -G(j\omega)^*
\end{pmatrix} \Pi(j\omega)
\begin{pmatrix}
I \\ -G(j\omega)
\end{pmatrix} & \prec 0  \\
\begin{pmatrix}
K(j\omega)^* & I
\end{pmatrix} \Pi(j\omega)
\begin{pmatrix}
K(j\omega) \\ I
\end{pmatrix} & \succeq 0.
\end{split}
\end{align}
First, a sufficiency condition for robust stability against phasal uncertainties is stated.

\begin{theorem}\label{thm:LTI_multiplier_phase_uncertainty_suff}
  Given $G \in \RR\Hinfty^{m \times n}$ and $K \in \RR\Hinfty^{n \times m}$, then $(I + u GK)^{-1} \in \RR\Hinfty^{m \times m}$ for all unitary
  $u \in \RR\Hinfty^{1 \times 1}$ if there exist $N \in \Linfty^{n \times n}$ and $M \in \Linfty^{m \times m}$ such that for all $\omega \in [0,
  \infty]$, $N(j\omega) = N(j\omega)^* \preceq 0$, $M(j\omega) = M(j\omega)^* \succeq 0$, and \eqref{eq: gain_sep} holds with
  \[
    \Pi:= \begin{bmatrix}
      N & 0 \\
      0 & M
\end{bmatrix}.
   \] 
 \end{theorem}

 \begin{proof}
By noting that $u(j\omega)^*u(j\omega) = 1$ for all $\omega \in [0, \infty]$, the claim holds by Proposition~\ref{prop:IQC_suff}. 
 \end{proof}

Next, a necessary condition for robust stability to phasal uncertainties is provided.
 
 \begin{theorem}\label{thm:LTI_multiplier_phase_uncertainty_nec}
  Given $G \in \RR\Hinfty^{m \times n}$ and $K \in \RR\Hinfty^{n \times m}$, then $(I + u GK)^{-1} \in \RR\Hinfty^{m \times m}$ for all unitary
  $u \in \RR\Hinfty^{1 \times 1}$ only if there exist $N \in \Linfty^{n \times n}$ and $M \in \Linfty^{m \times m}$ such that for all $\omega \in [0,
  \infty]$, $N(j\omega) = N(j\omega)^*$, $M(j\omega) = M(j\omega)^*$, and \eqref{eq: gain_sep} holds with
  \[
    \Pi:= \begin{bmatrix}
      N & 0 \\
      0 & M
\end{bmatrix}.
   \] 
 \end{theorem}

 \begin{proof}
The claim can be established by applying Theorem~\ref{thm:multiplier_phase_uncertainty} frequency-wise as in the proof for Proposition~\ref{prop:IQC_nec}.
 \end{proof}

 The theorem above shows that if a feedback system is robust against all scalar phasal uncertainties (with unity gain), then there necessarily exists
 a multiplier of the gain type with which to establish its robust stability via quadratic graph separation. When the phasic perturbations are allowed
 to span all present dimensions, the following necessary and sufficient condition for robust closed-loop stability may be acquired.

\begin{theorem}\label{thm:UGVK}
   Given $G, G^{-1}, K, K^{-1} \in \RR\Hinfty^{m \times m}$, then $(I + UGVK)^{-1} \in \RR\Hinfty^{m \times m}$ for all unitary
   $U, V \in \RR\Hinfty^{m \times m}$
   if and only if there exists $\gamma \in \Linfty^{1 \times 1}$
   such that for all $\omega \in [0, \infty]$, $|\gamma(j\omega)| > 0$ and \eqref{eq: gain_sep} holds with
  \[
    \Pi(j\omega):= \begin{bmatrix}
      -\xi |\gamma(j\omega)|^2 I & 0 \\
      0 & \xi I
\end{bmatrix},
\]
where $\xi \in \{-1, 1\}$.
 \end{theorem}

 \begin{proof}
Necessity can be established by applying Theorem~\ref{thm: UAVB} frequency-wise as in the proof for Proposition~\ref{prop:IQC_nec}. In particular,
continuity of $G$ and $K$ on the imaginary axis guarantees the uniqueness of $\xi$ for all $\omega \in [0, \infty]$. Sufficiency
follows from Proposition~\ref{prop:IQC_suff} and the fact that \eqref{eq: gain_sep} implies
\begin{align*}
  \begin{split}
\begin{pmatrix}
I & -G(j\omega)^*U(j\omega)^*
\end{pmatrix} \Pi(j\omega)
\begin{pmatrix}
I \\ -U(j\omega)G(j\omega)
\end{pmatrix} & \prec 0  \\
\begin{pmatrix}
K(j\omega)^*V(j\omega)^* & I
\end{pmatrix} \Pi(j\omega)
\begin{pmatrix}
V(j\omega)K(j\omega) \\ I
\end{pmatrix} & \succeq 0
\end{split}
\end{align*}
for all unitary $U \in \RR\Hinfty^{m \times m}$ and $V \in \RR\Hinfty^{m \times m}$.
\end{proof}

\begin{example}
Consider $G \in \RR\Hinfty^{m \times n}$ and $K \in \RR\Hinfty^{n \times m}$ for which $\|G\|_\infty < \gamma$ and $\|K\|_\infty \leq
\frac{1}{\gamma}$. Then $G$ and $K$ satisfy the quadratic separation condition in the theorems above, i.e., \eqref{eq: gain_sep}, with
  \[
    \Pi:= \begin{bmatrix}
      -\gamma^2 I & 0 \\
      0 & I
\end{bmatrix}.
\]
This is the celebrated small-gain theorem.
\end{example}

\section{A numerical example}\label{sec:num_ex}
In this section, we illustrate how the results in the paper can be used in practice. In particular, we here consider an example of verifying robust stability against positive scaling uncertainty, i.e., we focus on using Theorem~\ref{thm:LTI_magnitudinal_uncertainty}. It is noteworthy that the other results in this paper can be used analogously in appropriate settings.

To this end, consider the two transfer functions $G, K \in \RR\Hinfty^{2 \times 2}$ given in \eqref{eq:G_and_K}. The goal is to verify that the negative feedback interconnection between the two systems is stable against any positive scaling uncertainty, i.e., that $(I + \tau GK)^{-1} \in \RR\Hinfty^{2 \times 2}$ for all $\tau > 0$.
It can be easily seen that neither of the systems are passive, e.g., by noting that $G(j1) + G(j1)^* \not \succeq 0$ and that $K(j0) + K(j0)^* \not \succeq 0$,
and hence results based on passivity cannot be used to guarantee robust stability of the interconnection. However, from Theorem~\ref{thm:LTI_magnitudinal_uncertainty} we know that finding a phase type multiplier is a both necessary and sufficient condition for the sought robust stability.
To numerically obtain a stability certificate, we discretize the frequency interval into 1000 points $\{ \omega_\ell \}_{\ell = 1}^{1000}$; frequency $\omega_1 =  0$, and 999 grid points logarithmically equally spaced between $\omega_2 = 10^{-40}$ and $\omega_{1000} = {10^{40}}$ (note that $\omega_{501} = 1$). Then, for each $\omega_\ell$ we search for an $H_\ell \in \mM_2$ so that with
\[
  \Pi_\ell := \begin{bmatrix}
    0 & H_\ell \\
    H_\ell^* & 0
    \end{bmatrix},
\]
the two LMIs
\begin{align*}
\begin{pmatrix}
I & -G(j\omega_\ell)^*
\end{pmatrix} \Pi_\ell
\begin{pmatrix}
I \\ -G(j\omega_\ell)
\end{pmatrix} & \preceq -10^{-10} G(j\omega_\ell)^*G(j\omega_\ell)  \\
\begin{pmatrix}
K(j\omega_\ell)^* & I
\end{pmatrix} \Pi_\ell
\begin{pmatrix}
K(j\omega_\ell) \\ I
\end{pmatrix} & \succeq 0,
\end{align*}
are satisfied.
Here, we set $\epsilon = 10^{-10}$, since a small value of $\epsilon$ makes the LMIs easier to satisfy. If there exists one solution $H_\ell$ to the two LMIs, there in general exist multiple solutions. In order to obtain similar matrices for the different frequencies, for each $\ell$ we minimize $\| I_2 - H_\ell \|_2$ subject to the two LMIs as constraint. For each $\ell$, the resulting semidefinite programming problem is a convex optimization problem, and the numerical implementation is performed in Matlab using CVX \cite{cvx, grant2008graph} on a standard desktop computer with a 64-bit operating system (Windows 10), a 2.90GHz Intel i7-10700 CPU, and 32GB of RAM. The optimization problem is feasible for all $\ell = 1, \ldots, 1000$, which by continuity is a numerical certificate that the feedback interconnection of $G$ and $K$ is stable against any positive scaling uncertainty. Moreover, the total time to solve all the semi-definite programming problems was less than 5 minutes. Finally, for illustration purposes, a few of the obtained $H_\ell$s are given below.
\begin{align*}
\small
& H_1 \approx 10^{-7}\begin{bmatrix}
   7.82 &  -7.72 \\
   8.26 &  -8.12
\end{bmatrix},\\
& H_{501} \approx \begin{bmatrix}
   6.52 + 37.7j  & 4.60 + 52.9j \\
   1.53 + 53.1j  & 8.80 + 77.1j
\end{bmatrix},\\
& H_{1000} \approx I_2.
\end{align*}

\begin{figure*}[th]
\small
%
\hrulefill
\begin{equation}\label{eq:G_and_K}
G(s) = \frac{\begin{bmatrix} 7 s^2 + 57 s + 90 & 10 s^2 + 82 s + 132 \\ 10 s^2 + 73 s + 78 & 14 s^2 + 104 s + 120 \end{bmatrix}}{(s + 1)(s + 6)^2},
\quad
K(s) = \frac{-\begin{bmatrix} -7 s^3 - 91 s^2 - 413 s - 609 & 2 s^2 + 20 s + 58 \\ 5.5 s^3 + 71.5 s^2 + 324.5 s + 478.5 & 0.5 s^3 + 6 s^2 + 24.5 s + 29 \end{bmatrix}}{(s + 5 - 2j)^2 (s + 5 + 2j)^2}
\end{equation}
\hrulefill
\vspace*{4pt}
\end{figure*}

\section{Conclusions}\label{sec:conclusions}

We have shown that robustness of feedback interconnections against certain structured uncertainty corresponds to specific forms of quadratic
separation of the open-loop systems. Specifically, gain-type multipliers define quadratic separation needed in a feedback that is robust against all
phase-type uncertainty. Analogously, a robustly stable feedback against all gain-type uncertainty can always be established via the existence of phase-type multipliers. These results are importantly informative when using multiplier-based methods for establishing robust feedback stability. Future research directions of interest include the consideration of block-diagonal structured uncertainty as in the $\mu$-analysis and the investigation of its intricate relation with the main results in this paper. The exploration of a possibly
unifying description of the structures of uncertainties and the corresponding multipliers beyond those examined in this paper is also desirable.

\section*{Acknowledgement}
The authors would like to thank Chao Chen, Dan Wang, Di Zhao, and Ding Zhang for valuable discussions, and the anonymous reviewers for valuable feedback that helped improve the paper.

\appendix

\subsection{Proof of Theorem~\ref{thm:H_sectorial}, Corollary~\ref{cor:H_sectorial_A_inv}, and Corollary~\ref{cor: H_sectorial}}
\label{app:proof_thm:H_sectorial}

\begin{proof}[Proof of Theorem~\ref{thm:H_sectorial}]
The equivalence between (i) and (iii) is clear: the determinant is nonzero for all nonnegative scaling if and only if $AB$ have no eigenvalue along the strictly negative real axis.

Next, we prove that ``(ii) $\Rightarrow$ (i)''. To this end, assume that (ii) holds. By Lemma~\ref{lem: matrix_stability},
the existence of the multiplier $P$ that fulfills \eqref{eq:multipliers4} means that $\det(I + AB) \neq 0$. Now, for any $\tau \geq 0$ consider the matrices $\tilde{A} = A$ and $\tilde{B} = \tau B$. For these matrices, it is easily verified that this $P$ also fulfills \eqref{eq:multipliers4}. Therefore, by Lemma~\ref{lem: matrix_stability}
we have that $0 \neq \det(I + \tilde{A}\tilde{B}) = \det(I + \tau AB)$.

We complete the proof by showing that ``(iii) $\Rightarrow$ (ii)''. To this end, assume that \eqref{eq:condition}
holds. This means that the principle part of the matrix square root $(AB)^{1/2}$ is well-defined and that all the eigenvalues of $(AB)^{1/2}$ lie in the open right half-plane or at the origin \cite[Prob.~1.27]{higham2008functions}. Moreover, $(AB)^{1/2}$ has as many zero-eigenvalues as $AB$, and since a potential zero-eigenvalue of $AB$ is assumed to be semi-simple, so will the potential zero-eigenvalue of $(AB)^{1/2}$.
Next, let $\oplus$ denote the direct sum of two matrices, i.e., the block-diagonal with the two matrices on the diagonal \cite[Sec.~0.9.2]{horn2013matrix}, and
let $(AB)^{1/2} = X J X^{-1}$ be a Jordan normal form. More specifically, let $J = J_1 \oplus \cdots \oplus J_{\ell_1} \oplus \mathbf{0}$ where each block $J_k$ is of size $n_k$ and have the nonzero eigenvalue $\lambda_k$ on the diagonal, as per usual, but let the Jordan normal form be such that the elements on the sup-diagonal of each $J_k$ take the value 
\[
\epsilon = \min_{k \in \{ 1, \ldots, \ell_1 \}} \real\Big(\lambda_k\big((AB)^{1/2}\big)\Big)/2 > 0,
\]
which is always possible \cite[Cor.~3.1.21]{horn2013matrix}.

Now, set $D = J_1 \oplus \cdots \oplus J_{\ell_1} \oplus I$ and note that $D$ is of full rank, that $D^{-1} = J_1^{-1} \oplus \cdots \oplus J_{\ell_1}^{-1} \oplus I$, and that $D^{-1}J = I \oplus \cdots \oplus I \oplus \mathbf{0}$. Moreover, $D$ is strictly accretive. To see the latter, first note that by \cite[1.2.10, p.~12]{horn1994topics} we have that $W(D) = W(J_1 \oplus \cdots \oplus J_{\ell} \oplus I) = \text{Co}(W(J_1) \cup \cdots \cup W(J_{\ell}) \cup \{1 \})$, where $\text{Co}$ denotes the convex hull.
Moreover, $J_k = \lambda_k\big((AB)^{1/2}) I_{n_k \times n_k} + \epsilon S$, where $S$ is the nilepotent matrix with zeros everywhere except the first sup-diagonal which is ones. By \cite[1.2.10, p.~12]{horn1994topics} $W(J_k) \subset \lambda_k + \epsilon W(S)$, and by \cite[Prob.~29, pp.~45-46]{horn1994topics} the set $W(S)$ is contained in the unit disc. Since by construction $\epsilon \leq  \real(\lambda_k)/2$, we therefore have that $W(J_k) \subset \mC_+$ for all $k$, and hence $W(D) \subset \mC_+$, i.e., $D$ is strictly accretive. In particular, this means that $D + D^* \succ 0$.

Finally, take $H = A^*X^{-*} D^{-*} X^{-1}$ and note that
\[
HA + A^*H^* = A^{*}X^{-*}(D^{-*} + D^{-1})X^{-1}A \succeq \varepsilon A^*A
\]
for some $\varepsilon > 0$ small enough, since $X^{-*}(D^{-*} + D^{-1})X^{-1} \succ 0$. By multiplying the above inequality with $-1$, \eqref{eq:multipliers4_A} follows. Moreover, 
\begin{align*}
H^*B & = X^{-*} D^{-1} X^{-1} AB = X^{-*} D^{-1} X^{-1} X J X^{-1} X J X^{-1} \\
&= X^{-*} \big(I \oplus \cdots \oplus I \oplus \mathbf{0} \big) J X^{-1} = X^{-*} J X^{-1},
\end{align*}
which is congruent to $J$ and hence quasi-strictly accretive.
A direct calculation in \eqref{eq:multipliers4_B} therefore verifies that last claim, and hence proves that ``(iii) $\Rightarrow$ (ii)''.
\end{proof}

\begin{proof}[Proof of Corollary~\ref{cor:H_sectorial_A_inv}]
To prove the corollary, assume that $A \in \gln$. A direct calculation shows that if (iv) is fulfilled, then so is (ii).  Moreover, observe that the fact that $A$
is of full rank implies that $H = A^*X^{-*}D^{-*}X^{-1}$ constructed in the proof of Theorem~\ref{thm:H_sectorial}, above, is of full rank. Furthermore, it also means that
$HA + A^*H^* = A^{*}X^{-*}(D^{-*} + D^{-1})X^{-1}A$ is congruent to $D^{-*} + D^{-1}$ and therefore positive definite, i.e., $HA$ is strictly
accretive. In particular, that means that the corresponding multiplier $P$ fulfills \eqref{eq:multipliers}, and also that (iii) implies (iv) in this
case.

Finally, clearly (iv) implies (v), since a quasi-strictly accretive matrix is accretive. What is left to show is thus that under the assumption that $A \in \gln$, (v) implies any of the statements (i)-(iv). 
To this end, note that if $H \in \gln$ and $HA$ is strictly accreitve, then by congruence $HA$ and $AH^{-*}$ have the same phases. 
Therefore, using \cite[Lem.~2.4]{chen2024phase} we have that
\begin{align*}
-\pi & < \phimin(HA) + \phimin(H^*B) = \phimin(AH^{-*}) + \phimin(H^*B) \\
& \leq \angle \lambda_i(AH^{-*}H^{*}B) = \angle \lambda_i(AB) = \angle \lambda_i(AH^{-*}H^{*}B) \\
& \leq \phimax(HA) + \phimax(H^*B) < \pi
\end{align*}
for $i = 1, \ldots n$, which shows that (v) implies (iii).
\end{proof}

\begin{proof}[Proof of Corollary~\ref{cor: H_sectorial}]
Reexamining the proof of Theorem~\ref{thm:H_sectorial}, the proof of ``(i) $\Leftrightarrow$ (iii)'', and the proof of ``(ii) $\Rightarrow$ (i)'' hold directly also in the case of real matrices $A,B$ and $H$. Moreover, the remaining parts of the proof, showing that that ``(iii) $\Rightarrow$ (ii)'',  would also hold if the constructed $H$ is real. The latter is true if $X$ and $D$ are real, which is true if $(AB)^{1/2}$ is real. Thus, the conclusion follows if $AB$ has a real primary square root. Since a potential zero-eiganvalue is assumed to be semi-simple, by \cite[Thm.~1.23]{higham2008functions} the matrix $AB$ has a real primary square root.
\end{proof}

\subsection{Proof of Theorem~\ref{thm:TATstarSBSstar}}
\label{app:proof_thm:TATstarSBSstar}

The proof proceeds by showing that (i) and (iii) are equivalent, and that (ii) is equivalent to (iii). The former equivalence is the lengthier part, and for improved readability we hence separate the equivalence of (i) and (iii) into a separate proposition.

\begin{proposition}\label{lem:TATstarSBSstar}
Let $n \geq 2$, and let $A,B \in \mM_n \setminus \{ 0 \}$. Then
\begin{subequations}\label{eq:TATstarSBSstar_proof}
\begin{equation}
\det(I + T^*AT S^*BS) \neq 0 \text{ for all } T,S \in \gln,
\label{eq:TATstarSBSstar_proof_1}
\end{equation}
if and only if
\begin{align}
& \text{one matrix is quasi-sectorial, the other semi-sectorial, and} \nonumber\\
& \phimax(A) + \phimax(B) < \pi, \; \phimin(A) + \phimin(B) > -\pi.
\label{eq:TATstarSBSstar_proof_2}
\end{align}
\end{subequations}
\end{proposition}

\begin{proof}
To show $\Leftarrow$, assume that $\eqref{eq:TATstarSBSstar_proof_2}$ holds. 
For any $T,S \in \gln$, by congruence invariance of phases of matrices we have that $\phi(T^*AT) = \phi(A)$ and that $\phi(S^*BS) = \phi(B)$. Therefore, by \cite[Lem.~2.4]{chen2024phase} we have that
\begin{align*}
-\pi < & \phimin(A) + \phimin(B) = \phimin(T^*AT) + \phimin(S^*BS) \\
&  \leq \angle \lambda_i(T^*AT S^*BS) \leq \phimax(T^*AT) + \phimax(S^*BS) \\
& = \phimax(A) + \phimax(B) < \pi
\end{align*}
for $i = 1, \ldots n$.
In particular, this means that there exists an $\epsilon > 0$ so that $\lambda(T^*AT S^*BS) \cap \{ z \in \mC \mid z = -r e^{j\theta}, \; r > 0, \; \theta \in [-\epsilon, \epsilon] \} = \emptyset$ for all $T,S \in \gln$. The latter implies that \eqref{eq:TATstarSBSstar_proof_1} holds.

Next, to show $\Rightarrow$ we will show that the contraposition is true, namely that if \eqref{eq:TATstarSBSstar_proof_2} is not true, then there exists $T,S \in \gln$ such that $\det(I + T^*AT S^*BS) = 0$, i.e., such that $T^*AT S^*BS$ has an eigenvalue in $-1$. The latter is shown by explicitly considering all possible cases using the results in \cite{furtado2001spectral, furtado2003spectral}, and is also making heavy use of \cite[Thm.~1]{johnson1974field}  (see also \cite[Thm.~1.7.9]{horn1994topics}, \cite[Thm.~3]{furtado2001spectral}) .

To this end, first assume that $B$ is arbitrary and with at least one nonzero eigenvalue, and $A$ has only the zero-eigenvalue. Since $A \neq 0$, the
eigenvalue cannot be semisimple, and hence $A$ must have a Jordan block of size at least $2 \times 2$. The latter has a numerical range that is a
circle centered around the origin \cite[Prob.~9, pp.~25]{horn1994topics}, and hence the angular numerical range $2 \times 2$ block, and hence of the matrix $A$, is the entire complex plane.  Now, let
$B = V_B^* \Gamma_B V_B$ be a Schur decomposition of $B$, i.e., where $V_B$ is untiary and $\Gamma_B$ is upper triangular. Any such $\Gamma_B$ is
called a Schur form of $B$. Moreover, note that at least on element of $\Gamma_B$ is nonzero;
without loss of generality assume it is $(\Gamma_B)_{11}$. Next, by \cite[Thm.~1]{johnson1974field} there exists a $C \in \mM_n$ such that one of the eigenvalues of $C$ is $-(\Gamma_B)_{11}$ and such that $\tilde{T}^*A\tilde{T} = C$ for some $\tilde{T} \in \gln$. Moreover, let $C = V_C^* \Gamma_C V_C$ be a Schur decomposition of $C$ such that $(\Gamma_C)_{11} = -(\Gamma_B)_{11}$. By taking $S = V_B^*$ and $T = \tilde{T}V_C^*$ we have that
\begin{align*}
(T^*AT)(S^*BS) & = (V_C\tilde{T}^* A \tilde{T}V_C^*) (V_BV_B^*\Gamma_BV_BV_B^*) \\
& = (V_CCV_C^*) \Gamma_B = \Gamma_C \Gamma_B,
\end{align*}
which is upper triangular and with $-1$ in the upper left corner, i.e., for these $T$ and $S$ we have that $T^*AT S^*BS$ has an eigenvalue in $-1$.

Note that the above procedure can also be carried out, mutatis mutandis, if $B$ only has the zero-eigenvalue. In particular, if $B$ only has the zero-eigenvalue it must also have a Jordan block of size at least $2 \times 2$. Using \cite[Thm.~1]{johnson1974field},  by an appropriate selection of $S=S_1S_2$, we can thus make sure that $S_1^*BS_1$ has a nonzero eigenvalue, after which the above procedure can be repeated to select $T$ and $S_2$ so that $(T^*AT)(S^*BS)$ has an eigenvalue in $-1$. This means that in the following, we can always assume that both $A$ and $B$ have at least one nonzero eigenvalue.
In fact, for any matrix $A$ that is nonzero and which is not sectorial
and any arbitrary nonzero matrix $B$, a similar argument to the preceding one shows that $\det(I + T^*AT S^*BS) = 0$ for some $S,T \in \gln$, since the angular numerical range of $A$ is the entire complex plane (see \cite[Thm.~1]{johnson1974field}).

The above argument shows that a necessary condition for \eqref{eq:TATstarSBSstar_proof_1} to hold is that both $A,B$ are semi-sectorial. By \cite[Thm.~5]{furtado2003spectral} this means that, without loss of generality, we can restrict ourselves to consider matrices of the form
\begin{equation}\label{eq:TATSBS_form_of_A}
A = \! \left[ 
\begin{array}{c@{}cc}
 \!\!\! e^{j\theta_A} I_{k_2^A} \otimes
 \left[\begin{array}{cc}
         1 & 2 \\
         0 & 1 \\
  \end{array}\right] & \mathbf{0} & \mathbf{0} \! \\
  \mathbf{0} & \diag(e^{j\tilde{\phi}_1(A)}, \ldots, e^{j\tilde{\phi}_{k_1^A}(A)}) & \mathbf{0} \! \\
  \mathbf{0} & \mathbf{0} & \mathbf{0}\! \\
\end{array}\right],
\end{equation}
where $\theta_A + \pi/2 \geq \tilde{\phi}_1(A) \geq \ldots \geq \tilde{\phi}_{k_1^A}(A) \geq \theta_A - \pi/2$, $k_1^A \geq 0$, $k_2^A \geq 0$, and $n = 2k_2^A + k_1^A$, and analogously for $B$.
Note also that $\phimax(A) = \theta_A + \pi/2$ if $k_2^A > 0$ and $\phimax(A) = \tilde{\phi}_1(A)$ if $k_2^A = 0$; an analogous observation holds for $\phimin(A)$. Moreover a matrix $A$ of the form \eqref{eq:TATSBS_form_of_A} is quasi-sectorial if and only if $k_2^A = 0$ and $\phimax(A) - \phimin(A) < \pi$. Finally, note that by potentially applying an appropriate permutation that rearranges the block-diagonal elements we can, without loss of generality, restrict our attention to matrices of size $2 \times 2$.

Now, we first show that we cannot have $k_2^A > 0$ and $k_2^B > 0$.
To this end, let $S = I$ and consider the unitary matrix
\begin{align*}
T  & =
\begin{bmatrix}
\cos\left(\frac{\pi}{2} + \frac{\theta_A + \theta_B}{2} \right) & -\sin\left(\frac{\pi}{2} + \frac{\theta_A + \theta_B}{2} \right) \\
\sin\left(\frac{\pi}{2} + \frac{\theta_A + \theta_B}{2} \right) & \cos\left(\frac{\pi}{2} + \frac{\theta_A + \theta_B}{2} \right)
\end{bmatrix} \\
& = 
\begin{bmatrix}
-\sin\left(\frac{\theta_A + \theta_B}{2} \right) & -\cos\left(\frac{\theta_A + \theta_B}{2} \right) \\
\cos\left(\frac{\theta_A + \theta_B}{2} \right) & -\sin\left(\frac{\theta_A + \theta_B}{2} \right)
\end{bmatrix}.
\end{align*}
Let $S_{AB} := \sin((\theta_A + \theta_B)/2)$ and $C_{AB} := \cos((\theta_A + \theta_B)/2)$. A direct (albeit somewhat cumbersome) calculation gives that 
\begin{align*}
& T^*\begin{bmatrix}
1 & 2 \\
0 & 1
\end{bmatrix}
T\begin{bmatrix}
1 & 2 \\
0 & 1
\end{bmatrix}
= \\
& \begin{bmatrix}
(S_{AB} - C_{AB})^2 & 4S_{AB}^2 - 4S_{AB}C_{AB} + 2C_{AB}^2 \\
-2C_{AB}^2     & S_{AB}^2 + 2S_{AB}C_{AB} - 3C_{AB}^2
\end{bmatrix},
\end{align*}
which has eigenvalues $- \cos(\theta_A + \theta_B) \pm i \sin(\theta_A + \theta_B) = -e^{\mp i (\theta_A + \theta_B)}$. Therefore, taking $T$ as above and $S = I$, the matrix $T^*AT S^*BS$ has an eigenvalue in $-1$.

Next, we therefore assume that $k_2^A > 0$ and $k_2^B = 0$. In this case, first assume that $B$ only has one non-zero
phase, in which case it suffices to consider
\[
A = e^{j\theta_A}
\begin{bmatrix}
1 & 2 \\
0 & 1
\end{bmatrix}
\qquad \text{and} \qquad B = \diag(e^{j\phi_1(B)}, 0).
\]
If $\theta_A + \pi/2 + \phi_1(B) \geq \pi$ or $\theta_A - \pi/2 + \phi_1(B) \leq -\pi$, then we can write
\[
T^*ATB = T^*\tilde{A}T \diag(1, 0)
\]
where $\tilde{A} = e^{j\phi_1(B)}A$. However, since $\theta_A + \pi/2 + \phi_1(B) \geq \pi$ or $\theta_A - \pi/2 + \phi_1(B) \leq -\pi$, $-1$ is in the numerical range of $\tilde{A}$. Therefore, using \cite[Thm.~1]{johnson1974field} we can make a construction similar to before, and select an appropriate $T$ such that $T^*ATB$ has an eigenvalue in $-1$. On the other hand, if  $\theta_A + \pi/2 + \phi_1(B) < \pi$ and $\theta_A - \pi/2 + \phi_1(B) > -\pi$, then \eqref{eq:TATstarSBSstar_proof_2} is fulfilled (and thus \eqref{eq:TATstarSBSstar_proof_1} holds, see the proof of the implication ``$\Leftarrow$'').

The next case we consider is when the diagonal unitary part of $B$ is of size at least $2 \times 2$. To this end, it suffices the consider
\[
A = e^{j\theta_A}
\begin{bmatrix}
1 & 2 \\
0 & 1
\end{bmatrix}
\qquad \text{and} \qquad B = \diag(e^{j\phimax(B)}, e^{j\phimin(B)}).
\]
We split this into two different subcases.  In the first case, assume that $B$ is quasi-sectorial, which means that $\phimax(B) - \phimin(B) < \pi$.
If $\theta_A + \pi/2 + \phimax(B) \geq \pi$ or $\theta_A - \pi/2 + \phimin(B) \leq -\pi$, then we can make constructions analogous to the above one, and if  $\theta_A + \pi/2 + \phimax(B) < \pi$ and $\theta_A - \pi/2 + \phimin(B) > -\pi$, then \eqref{eq:TATstarSBSstar_proof_2} is fulfilled and thus \eqref{eq:TATstarSBSstar_proof_1} holds.
Therefore, we next assume that $\phimax(B) - \phimin(B) = \pi$, in which case $B$ is rotation-Hermitian, i.e., $B = e^{j\phimax(B)} \diag(1, -1)$. Moreover, that means that either $\theta_A + \pi/2 + \phimax(B) \geq \pi$ or $\theta_A - \pi/2 + \phimin(B) = \theta_A - \pi/2 + \phimax(B) - \pi \leq -\pi$. In any case, let $S = I$ and let 
\[
T = \begin{bmatrix}
\cos\left(\frac{\theta_A + \phimax(B)}{2}\right) & j \sin\left(\frac{\theta_A + \phimax(B)}{2}\right) \\
j \sin\left(\frac{\theta_A + \phimax(B)}{2}\right) & \cos\left(\frac{\theta_A + \phimax(B)}{2}\right)
\end{bmatrix}.
\]
This $T$ is unitary, and a direct (albeit somewhat cumbersome) calculation verify that $T^* A T B$ has an eigenvalue in $-1$.
This means that we cannot have  $k_2^A > 0$ and $k_2^B = 0$.

Now, consider the case where both $A$ and $B$ have a rotation-Hermitian $2\times2$ block, i.e., when
\[
A = e^{j\phimax(A)}
\begin{bmatrix}
1 & 0 \\
0 & -1
\end{bmatrix},
\qquad
B = e^{j\phimax(B)}
\begin{bmatrix}
1 & 0 \\
0 & -1
\end{bmatrix}.
\]
Similarly to the last case just above, that means that either $\phimax(A) + \phimax(B) \geq \pi$ or $\phimax(A) - \pi + \phimax(B) -\pi \leq -\pi$.
In any case, let $S = I$ and let
\[
T = 
\begin{bmatrix}
\cos\left(\frac{\pi}{2} - \frac{\phimax(A) + \phimax(B)}{2} \right) & -\sin\left(\frac{\pi}{2} - \frac{\phimax(A) + \phimax(B)}{2} \right) \\
\sin\left(\frac{\pi}{2} - \frac{\phimax(A) + \phimax(B)}{2} \right) & \cos\left(\frac{\pi}{2} - \frac{\phimax(A) + \phimax(B)}{2} \right)
\end{bmatrix}.
\]
A calculation similar to before shows that $T^*ATB$ has an eigenvalue in $-1$.

The two final cases to consider is when either i) $A$ has a rotation-Hermitian $2 \times 2$ block and $B$ is quasi-sectorial, or ii) when both $A$ and $B$ are quasi-sectorial, but when the phase condition is not satisfied in either case. The two cases can be handled together, and we can, without loss of generality, assume that $\phimax(A) + \phimax(B) \geq \pi$. In this case, by \cite[Thm.~1]{johnson1974field} there is a $T$ such that $C = T^*AT$ has an eigenvalue in $e^{j\phimax(A)}$. Let $C = V_C^* \Gamma_C V_C$ be a Schur decomposition, with $e^{j\phimax(A)}$ as top-left element. Let $S = S_1 S_2$, and note that since $\phimax(B) \geq \pi - \phimax(A)$ and $B$ is quasi-sectorial we can in a similar way select $S_1$ so that $D = S_1^* B S_1$ has an eigenvalue in $e^{j(\pi - \phimax(A))}$. Let $D = V_D^* \Gamma_D V_D$ be a Schur decomposition with $e^{j(\pi - \phimax(A))}$ as top-left element. By taking $S_2 = V_D^*V_C$, we get
\begin{align*}
& (T^*AT) (SBS^*) = C (S_2^* D S_2) \\
& = V_C^* \Gamma_C V_C V_C^* V_D  V_D^* \Gamma_D V_D V_D^* V_C = V_C^* \Gamma_C \Gamma_D V_C,
\end{align*}
which by construction has an eigenvalue in $-1$.

In summary, this means that unless \eqref{eq:TATstarSBSstar_proof_2} holds, then there exist $T,S \in \gln$ such that $\det(I + T^*AT S^*BS) = 0$. This shows the implication $\Rightarrow$, and hence the result follows.
\end{proof}

\begin{proof}[Proof of Theorem~\ref{thm:TATstarSBSstar}]
For $n = 1$ the matrices are scalar and hence commute. Therefore, in this case $\det(I + T^*AT S^*BS) = \det(1 + \tau a b)$ for $\tau > 0$, and the conclusions follow almost trivially. 

For $n\geq 2$, Proposition~\ref{lem:TATstarSBSstar} shows that (i) and (iii) are equivalent. Next, we prove that 
``(iii) $\Rightarrow$ (ii)''. To this end, without loss of generality, assume that $A$ is quasi-sectorial of rank $n-k$, and that $B$ is semi-sectorial. The fact that the sum of the largest and smallest phases are bounded away from $\pm \pi$, respectively, implies that there exists a $z \in \mT$ such that $zA$ is quasi-strictly accretive and $z^*B$ is accretive, and hence in particular that $-zA -z^*A^* \preceq 0$ and $z^*B + zB^* \succeq 0$. The latter means that for this $P$, \eqref{eq:multipliers4_B} holds. It remains to show that  \eqref{eq:multipliers4_A} holds, i.e., that the former inequality above can be strengthened to $-zA -z^*A^* \preceq -\varepsilon A^*A$ for some $\varepsilon > 0$. To do so, let $zA = T^*DT$ be a sectorial decomposition of the quasi-strictly accretive $zA$. In particular, this means that
\[
D = \diag(e^{j\phi_1}, \ldots, e^{j \phi_{n-k}}, \underbrace{0, \ldots, 0}_{k \text{ of them}}),
\]
and $zA + z^*A^* = T^* (D + D^*) T$, which is positive semi-definite with the top-left block of $D + D^*$ containing the $n-k$ strictly positive eigenvalues. A direct calculation
gives that
\[
A^*A = A^*z^*zA = T^*D^*TT^*DT = T^*
\begin{bmatrix}
\star      & 0 \\
0          & 0
\end{bmatrix} T =: T^* \Delta T,
\]
where the block $\star$ is of dimension $n-k \times n-k$ and is positive definite. In particular, this means that for $\epsilon > 0$ small enough we have that $D + D^* \succ \epsilon \Delta$, and hence that
\[
zA + z^*A^* = T^* (D + D^*) T \succ \epsilon T^* \Delta T = \epsilon A^*A.
\]
Multiplying the above inequality by $-1$ gives the inequality \eqref{eq:multipliers4_A}. This completes the proof of the implication ``(iii) $\Rightarrow$ (ii)''.

To show that ``(ii) $\Rightarrow$ (iii)'', without loss of generality, assume that $P$ fulfills \eqref{eq:multipliers4}. The proof for the case where
$P$ fulfills \eqref{eq:multipliers3} is analogous. Now, note that \eqref{eq:multipliers4} implies that both $zA$ and $z^*B$ are accretive. What
remains to be shown is thus that $zA$ is in fact quasi-strictly accretive. To this end, let $zA = T^*DT$ be the sectorial decomposition, with $D$ of
the form \eqref{eq:TATSBS_form_of_A}. By an argument similar to the one above, we have that \eqref{eq:multipliers4_A} implies that
\[
-\begin{bmatrix}
\ast      & 0 \\
0          & 0
\end{bmatrix} =
-D -D^* \preceq - \varepsilon D^* T T^* D = -\varepsilon
\begin{bmatrix}
\star      & 0 \\
0          & 0
\end{bmatrix},
\]
where $\ast$ is block-diagonal and positive semi-definite, and $\star$ is positive definite. However, the inequality means that $\ast \succeq \varepsilon \star$ for some $\varepsilon > 0$, and since $\star$ is positive definite this can only be true if $\ast$ is also positive definite. Now, if $D$ has a block
\[
\begin{bmatrix}
1 & 2\\
0 & 1
\end{bmatrix}
\]
a direct calculation gives that $\ast$ contains a block 
\[
\begin{bmatrix}
2 & 2\\
2 & 2
\end{bmatrix}.
\]
This would mean that $\ast$ is not of full rank, and hence
it is only positive semi-definite. Therefore, $D$ cannot contain any such blocks, which implies that $zA$ is quasi-strictly accretive.

Finally, the last part of the theorem follows by simply reexamining the proofs for the equivalence of (ii) and (iii) under the additional
assumption that $A$ is of full rank. It is then easily seen that the same conclusion holds, but with \eqref{eq:multipliers4} replaced by \eqref{eq:multipliers}.
\end{proof}

\subsection{Proof of Theorem~\ref{thm: UAVB}}
\label{app:proof_thm: UAVB}

\begin{proof}[Proof of Theorem~\ref{thm: UAVB}]
The equivalence between (ii) and (iii) follows directly from Lemma~\ref{lem:multiplier_and_singular_values}.

To show that ``(ii) $\Rightarrow$ (i)'', first assume that there exists a multiplier $P$ of the form \eqref{eq:P_gain}, with $\xi = 1$, that satisfies \eqref{eq:multipliers}.
A direct calculation, as in the proof of Lemma~\ref{lem:multiplier_and_singular_values}, gives that \eqref{eq:UAVB_proof} holds.
Now, let $U \in \mU_m$ and $V \in \mU_n$, and note that for $\tilde{A} = UA$ and $\tilde{B} = VB$ we have
\begin{align*}
& \tilde{A}^*\tilde{A} = A^*U^*UA = A^*A \prec \gamma^2 I, \\
& \gamma^2 \tilde{B}^*\tilde{B} = \gamma^2 B^*V^*VB = \gamma^2 B^*B \prec I.
\end{align*}
Thus, for this $P$, \eqref{eq:multipliers} holds for $\tilde{A}$ and $\tilde{B}$ and hence $0 \neq \det(I + \tilde{A}\tilde{B}) = \det(I + UAVB)$ by Lemma~\ref{lem: matrix_stability}. Since $U \in \mU_m$ and $V \in \mU_n$ were arbitrary, the implication follows in the case of $\xi = 1$. The proof for the case $\xi = -1$ follows analogously.  

We now show that ``(i) $\Rightarrow$ (iii)''. To this end, assume that $\det(I + UAVB) \neq 0$ for all $U \in \mU_m$ and all $V \in \mU_n$.
First note that the statement is trivial if any of the two matrices $A$ and $B$ is the zero matrix and hence we can, without loss of generality, assume that neither of them is. Now, let $A = W_A \Sigma_A V_A^*$ and $B = W_B \Sigma_B V_B^*$ be the singular value decompositions of $A$ and $B$, respectively, where
$W_A, V_B \in \mU_m$,
$V_A, W_B \in \mU_n$,
$\Sigma_A \in \mM_{m,n}$, and
$\Sigma_B \in \mM_{n, m}$. 
Next, note that for any $k \geq 1$, $\mU_k$ is closed under matrix multiplication, i.e., that for all $U,V \in \mU_n$, $UV \in \mU_n$,
and that all permutation matrices are unitary. Therefore, for any $\tilde{V}, \tilde{W} \in \mU_m$ and any permutation matrix $\mathcal{P} \in \mU_n$, let $U = V_B \tilde{V}^* \tilde{W} W_A^* \in \mU_m$ and $V = V_A \mathcal{P} W_B^* \in \mU_n$. This means that
\begin{align*}
0 & \neq \det(I + UAVB) =  \det(I + V_B \tilde{V}^*\tilde{W}\Sigma_A \mathcal{P} \Sigma_B V_B^*) \\
&  = \det(I + \tilde{W}\Sigma_A \mathcal{P} \Sigma_B \tilde{V}^*)
\end{align*}
for all $\tilde{V}, \tilde{W} \in \mU_n$ and all permutation matrices $\mathcal{P}$.

Next, assume that $n \geq m$.
In this case, note that $\Sigma_A \mathcal{P} \in \mM_{m,n}$ with the $n$ columns of $\Sigma_A$ permuted according to the permutation matrix $\mathcal{P}$. Therefore, $\tilde{W}\Sigma_A \mathcal{P} \Sigma_B \tilde{V}^*$ can be identified as a singular value decomposition of the matrix whose singular values are given by $\sigma_{\Phi(i)}(A)\sigma_{i}(B)$, $i = 1,\ldots,m$, where $\Phi : \{1, \ldots, m\} \mapsto \{1, \ldots, n\}$ is the injective map corresponding to the permutation matrix $\mathcal{P}$. Moreover, by appropriately selecting  the permutation matrix $\mathcal{P}$ we can get any injective map that maps from $\{1, \ldots, m\}$ to $\{1, \ldots, n\}$.
Now, if there exists a $\Phi$ such that $\max_{k} \sigma_{\Phi(k)}(A)\sigma_k(B) \geq 1$ and $\min_{k} \sigma_{\Phi(k)}(A)\sigma_k(B) \leq 1$, then by \cite{horn1954eigenvalues} (see also \cite[Thm.~9.E.5]{marshall2011inequalities}) there exist matrices $\tilde{W}, \tilde{V} \in \mU_n$ such that the corresponding matrix $\tilde{W}\Sigma_A P \Sigma_B \tilde{V}^*$ has an eigenvalue in $-1$. However, that would mean that the corresponding determinant is zero, which is a contradiction.
Therefore, for all $\Phi$ we must either have that $\max_{k} \sigma_{\Phi(k)}(A)\sigma_k(B) < 1$ or that $\min_{k} \sigma_{\Phi(k)}(A) \sigma_k(B) > 1$. In particular, this must hold for all $\Phi$ such that $\Phi(1) = 1$ and $\Phi(m) = n$, in which case either $1 > \max_{k} \sigma_{\Phi(k)}(A)\sigma_k(B) = \sigma_{1}(A)\sigma_1(B)$ or $1 < \min_{k} \sigma_{\Phi(k)}(A)\sigma_k(B) = \sigma_{n}(A) \sigma_m(B)$. This shows that the (i) implies (iii) in the case where $n \geq m$.

To complete the proof of the theorem, assume that $m > n$.
In this case $\tilde{W}\Sigma_A \mathcal{P} \Sigma_B \tilde{V}^*$ can still be identified as a singular value decomposition of the matrix, but now the singular values are given by $m-n$ zeros as well as $\sigma_{\Phi(i)}(A)\sigma_{i}(B)$, $i = 1,\ldots,n$, where $\Phi : \{1, \ldots, n\} \mapsto \{1, \ldots, n\}$ is a permutation. This means that $0$ will always be a singular value of $\tilde{W}\Sigma_A \mathcal{P} \Sigma_B \tilde{V}^*$, and by arguments similar to those in the previous paragraph we must therefore have that for all permutations $\Phi$ it holds that $\max_{k} \sigma_{\Phi(k)}(A)\sigma_k(B) < 1$, and hence in particular that $\sigma_{1}(A)\sigma_1(B) < 1$.
\end{proof}


\balance

\bibliographystyle{IEEEtran}
\bibliography{}

\begin{thebibliography}{10}
\providecommand{\url}[1]{#1}
\csname url@samestyle\endcsname
\providecommand{\newblock}{\relax}
\providecommand{\bibinfo}[2]{#2}
\providecommand{\BIBentrySTDinterwordspacing}{\spaceskip=0pt\relax}
\providecommand{\BIBentryALTinterwordstretchfactor}{4}
\providecommand{\BIBentryALTinterwordspacing}{\spaceskip=\fontdimen2\font plus
\BIBentryALTinterwordstretchfactor\fontdimen3\font minus
  \fontdimen4\font\relax}
\providecommand{\BIBforeignlanguage}[2]{{%
\expandafter\ifx\csname l@#1\endcsname\relax
\typeout{** WARNING: IEEEtran.bst: No hyphenation pattern has been}%
\typeout{** loaded for the language `#1'. Using the pattern for}%
\typeout{** the default language instead.}%
\else
\language=\csname l@#1\endcsname
\fi
#2}}
\providecommand{\BIBdecl}{\relax}
\BIBdecl

\bibitem{teel1996input}
A.~R. Teel, T.~T. Georgiou, L.~Praly, and E.~D. Sontag, ``Input-output
  stability,'' in \emph{The Control Systems Handbook : Control System Advanced
  Methods}, 2nd~ed., W.~S. Levine, Ed.\hskip 1em plus 0.5em minus 0.4em\relax
  Boca Raton, FL: CRC Press, 2011, pp. 44.1--44.23.

\bibitem{teel1996graphs}
A.~R. Teel, ``On graphs, conic relations, and input-output stability of
  nonlinear feedback systems,'' \emph{IEEE Transactions on Automatic Control},
  vol.~41, no.~5, pp. 702--709, 1996.

\bibitem{megretski2010integral}
A.~Megretski, U.~T. J{\"o}nsson, C.-Y. Kao, and A.~Rantzer, ``Integral
  quadratic constraints,'' in \emph{The Control Systems Handbook : Control
  System Advanced Methods}, 2nd~ed., W.~S. Levine, Ed.\hskip 1em plus 0.5em
  minus 0.4em\relax Boca Raton, FL: CRC Press, 2011, pp. 41.1--41.19.

\bibitem{zames1966input}
G.~Zames, ``On the input-output stability of time-varying nonlinear feedback
  systems part one: Conditions derived using concepts of loop gain, conicity,
  and positivity,'' \emph{IEEE Transactions on Automatic Control}, vol.~11,
  no.~2, pp. 228--238, 1966.

\bibitem{zames1968stability}
G.~Zames and P.~Falb, ``Stability conditions for systems with monotone and
  slope-restricted nonlinearities,'' \emph{SIAM Journal on Control}, vol.~6,
  no.~1, pp. 89--108, 1968.

\bibitem{megretski1997system}
A.~Megretski and A.~Rantzer, ``System analysis via integral quadratic
  constraints,'' \emph{IEEE Transactions on Automatic Control}, vol.~42, no.~6,
  pp. 819--830, 1997.

\bibitem{lessard2016analysis}
L.~Lessard, B.~Recht, and A.~Packard, ``Analysis and design of optimization
  algorithms via integral quadratic constraints,'' \emph{SIAM Journal on
  Optimization}, vol.~26, no.~1, pp. 57--95, 2016.

\bibitem{hu2016exponential}
B.~Hu and P.~Seiler, ``Exponential decay rate conditions for uncertain linear
  systems using integral quadratic constraints,'' \emph{IEEE Transactions on
  Automatic Control}, vol.~61, no.~11, pp. 3631--3637, 2016.

\bibitem{khong2021iqc}
S.~Z. Khong, ``On integral quadratic constraints,'' \emph{IEEE Transactions on
  Automatic Control}, vol.~67, no.~3, pp. 1603--1608, 2022.

\bibitem{khong2024connections}
S.~Z. Khong and A.~Lanzon, ``Connections between integral quadratic constraints
  and dissipativity,'' \emph{IEEE Transactions on Automatic Control}, vol.~69,
  no.~8, pp. 5672--5677, 2024.

\bibitem{cantoni2011robustness}
M.~Cantoni, U.~T. J{\"o}nsson, and C.-Y. Kao, ``Robustness analysis for
  feedback interconnections of distributed systems via integral quadratic
  constraints,'' \emph{IEEE Transactions on Automatic Control}, vol.~57, no.~2,
  pp. 302--317, 2011.

\bibitem{cantoni2013robust}
M.~Cantoni, U.~T. J{\"o}nsson, and S.~Z. Khong, ``Robust stability analysis for
  feedback interconnections of time-varying linear systems,'' \emph{SIAM
  Journal on Control and Optimization}, vol.~51, no.~1, pp. 353--379, 2013.

\bibitem{khong2016robust}
S.~Z. Khong, E.~Lovisari, and C.-Y. Kao, ``Robust synchronization in
  multi-agent networks with unstable dynamics,'' \emph{IEEE Transactions on
  Control of Network Systems}, vol.~5, no.~1, pp. 205--214, 2016.

\bibitem{khong2025feedback}
S.~Z. Khong and A.~Lanzon, ``Feedback stability analysis via frequency
  dependent constraints,'' \emph{IEEE Transactions on Automatic Control},
  vol.~70, no.~2, pp. 1228--1235, 2025.

\bibitem{megretski1993power}
A.~Megretski and S.~Treil, ``Power distribution inequalities in optimization
  and robustness of uncertain systems,'' \emph{Journal of Mathematical Systems,
  Estimation, and Control}, vol.~3, no.~3, pp. 301--319, 1993.

\bibitem{zhou1996robust}
K.~Zhou, J.~C. Doyle, and K.~Glover, \emph{Robust and optimal control}.\hskip
  1em plus 0.5em minus 0.4em\relax Upper Saddle River, NJ: Prentice-Hall, 1996.

\bibitem{georgiou1990optimal}
T.~T. Georgiou and M.~C. Smith, ``Optimal robustness in the gap metric,''
  \emph{IEEE Transactions on Automatic Control}, vol.~35, no.~6, pp. 673--686,
  1990.

\bibitem{qui1992feedback}
L.~Qiu and E.~J. Davison, ``Feedback stability under simultaneous gap metric
  uncertainties in plant and controller,'' \emph{Systems \& Control Letters},
  vol.~18, no.~1, pp. 9--22, 1992.

\bibitem{zhao2020stabilization}
D.~Zhao, L.~Qiu, and G.~Gu, ``Stabilization of two-port networked systems with
  simultaneous uncertainties in plant, controller, and communication
  channels,'' \emph{IEEE Transactions on Automatic Control}, vol.~65, no.~3,
  pp. 1160--1175, 2020.

\bibitem{willems1968some}
J.~C. Willems and R.~Brockett, ``Some new rearrangement inequalities having
  application in stability analysis,'' \emph{IEEE Transactions on Automatic
  Control}, vol.~13, no.~5, pp. 539--549, 1968.

\bibitem{jonsson2000optimization}
U.~T. J{\"o}nsson and A.~Rantzer, ``Optimization of integral quadratic
  constraints,'' in \emph{Advances in Linear Matrix Inequality Methods in
  Control}.\hskip 1em plus 0.5em minus 0.4em\relax SIAM, 2000, pp. 109--127.

\bibitem{scherer2001lpv}
C.~W. Scherer, ``{LPV} control and full block multipliers,'' \emph{Automatica},
  vol.~37, no.~3, pp. 361--375, 2001.

\bibitem{kao2007stability}
C.-Y. Kao and A.~Rantzer, ``Stability analysis of systems with uncertain
  time-varying delays,'' \emph{Automatica}, vol.~43, no.~6, pp. 959--970, 2007.

\bibitem{pfifer2015integral}
H.~Pfifer and P.~Seiler, ``Integral quadratic constraints for delayed nonlinear
  and parameter-varying systems,'' \emph{Automatica}, vol.~56, pp. 36--43,
  2015.

\bibitem{iwasaki1998well}
T.~Iwasaki and S.~Hara, ``Well-posedness of feedback systems: Insights into
  exact robustness analysis and approximate computations,'' \emph{IEEE
  Transactions on Automatic Control}, vol.~43, no.~5, pp. 619--630, 1998.

\bibitem{doyle1982analysis}
J.~Doyle, ``Analysis of feedback systems with structured uncertainties,'' in
  \emph{IEE Proceedings D (Control Theory and Applications)}, vol. 129,
  no.~6.\hskip 1em plus 0.5em minus 0.4em\relax Institution of Electrical
  Engineers, 1982, pp. 242--250.

\bibitem{safonov1982stability}
M.~G. Safonov, ``Stability margins of diagonally perturbed multivariable
  feedback systems,'' in \emph{IEE Proceedings D (Control Theory and
  Applications)}, vol. 129, no.~6.\hskip 1em plus 0.5em minus 0.4em\relax IET,
  1982, pp. 251--256.

\bibitem{packard1993complex}
A.~Packard and J.~Doyle, ``The complex structured singular value,''
  \emph{Automatica}, vol.~29, no.~1, pp. 71--109, 1993.

\bibitem{nemirovskii1993several}
A.~Nemirovskii, ``Several {NP}-hard problems arising in robust stability
  analysis,'' \emph{Mathematics of Control, Signals and Systems}, vol.~6, pp.
  99--105, 1993.

\bibitem{braatz1994computational}
R.~P. Braatz, P.~M. Young, J.~C. Doyle, and M.~Morari, ``Computational
  complexity of $\mu$ calculation,'' \emph{IEEE Transactions on Automatic
  Control}, vol.~39, no.~5, pp. 1000--1002, 1994.

\bibitem{fan1991robustness}
M.~K. Fan, A.~L. Tits, and J.~C. Doyle, ``Robustness in the presence of mixed
  parametric uncertainty and unmodeled dynamics,'' \emph{IEEE Transactions on
  Automatic Control}, vol.~36, no.~1, pp. 25--38, 1991.

\bibitem{meinsma1997dual}
G.~Meinsma, Y.~Shrivastava, and M.~Fu, ``A dual formulation of mixed $\mu$ and
  on the losslessness of $({D}, {G})$ scaling,'' \emph{IEEE Transactions on
  Automatic Control}, vol.~42, no.~7, pp. 1032--1036, 1997.

\bibitem{chellaboina2008structured}
V.~Chellaboina, W.~M. Haddad, and A.~Kamath, ``The structured phase margin for
  robust stability analysis of linear systems with phase and time delay
  uncertainties,'' \emph{International Journal of Control}, vol.~81, no.~8, pp.
  1298--1310, 2008.

\bibitem{chou1999stability}
Y.-S. Chou, A.~L. Tits, and V.~Balakrishnan, ``Stability multipliers and $\mu$
  upper bounds: connections and implications for numerical verification of
  frequency domain conditions,'' \emph{IEEE Transactions on Automatic Control},
  vol.~44, no.~5, pp. 906--913, 1999.

\bibitem{hall1993mixed}
S.~R. Hall and J.~P. How, ``Mixed $\mathcal{H}_2$/$\mu$ performance bounds
  using dissipation theory,'' in \emph{Proceedings of 32nd IEEE Conference on
  Decision and Control}.\hskip 1em plus 0.5em minus 0.4em\relax IEEE, 1993, pp.
  1536--1541.

\bibitem{how1993connections}
J.~P. How and S.~R. Hall, ``Connections between the {P}opov stability criterion
  and bounds for real parameter uncertainty,'' in \emph{1993 American Control
  Conference}.\hskip 1em plus 0.5em minus 0.4em\relax IEEE, 1993, pp.
  1084--1089.

\bibitem{fu1997improved}
M.~Fu and N.~E. Barabanov, ``Improved upper bounds for the mixed structured
  singular value,'' \emph{IEEE Transactions on Automatic Control}, vol.~42,
  no.~10, pp. 1447--1452, 1997.

\bibitem{khong2018converse}
S.~Z. Khong and A.~van~der Schaft, ``On the converse of the passivity and
  small-gain theorems for input--output maps,'' \emph{Automatica}, vol.~97, pp.
  58--63, 2018.

\bibitem{khong2021converse}
S.~Z. Khong and C.-Y. Kao, ``Converse theorems for integral quadratic
  constraints,'' \emph{IEEE Transactions on Automatic Control}, vol.~66, no.~8,
  pp. 3695--3701, 2021.

\bibitem{khong2022converse}
------, ``Addendum to ``converse theorems for integral quadratic
  constraints",'' \emph{IEEE Transactions on Automatic Control}, vol.~67,
  no.~1, pp. 539--540, 2022.

\bibitem{khong2023converse}
S.~Z. Khong, D.~Zhao, and A.~Lanzon, ``Converse negative imaginary theorems,''
  \emph{Automatica}, vol. 165, p. 111682, 2024.

\bibitem{horn2013matrix}
R.~A. Horn and C.~R. Johnson, \emph{Matrix Analysis}, 2nd~ed.\hskip 1em plus
  0.5em minus 0.4em\relax New York, NY: Cambridge University Press, 2013.

\bibitem{foias1993robust}
C.~Foias, T.~T. Georgiou, and M.~C. Smith, ``Robust stability of feedback
  systems: A geometric approach using the gap metric,'' \emph{SIAM Journal on
  Control and Optimization}, vol.~31, no.~6, pp. 1518--1537, 1993.

\bibitem{doyle1993parallel}
J.~C. Doyle, T.~T. Georgiou, and M.~C. Smith, ``The parallel projection
  operators of a nonlinear feedback system,'' \emph{Systems \& Control
  Letters}, vol.~20, no.~2, pp. 79--85, 1993.

\bibitem{vinnicombe2001uncertainty}
G.~Vinnicombe, \emph{Uncertainty and Feedback -- $H_{\infty}$ Loop-shaping and
  the $\nu$-gap metric}.\hskip 1em plus 0.5em minus 0.4em\relax London:
  Imperial College Press, 2001.

\bibitem{zahedzadeh2008input}
V.~Zahedzadeh, H.~J. Marquez, and T.~Chen, ``On the input-output stability of
  nonlinear systems: Large gain theorem,'' in \emph{2008 American Control
  Conference}.\hskip 1em plus 0.5em minus 0.4em\relax IEEE, 2008, pp.
  3440--3445.

\bibitem{horn1994topics}
R.~A. Horn and C.~R. Johnson, \emph{Topics in Matrix Analysis}.\hskip 1em plus
  0.5em minus 0.4em\relax New York, NY: Cambridge University Press, 1994.

\bibitem{zhang2011matrix}
F.~Zhang, \emph{Matrix Theory: Basic Results and Techniques}.\hskip 1em plus
  0.5em minus 0.4em\relax New York, NY: Springer, 2011.

\bibitem{gustafson1997numerical}
K.~E. Gustafson and D.~K. Rao, \emph{Numerical Range}.\hskip 1em plus 0.5em
  minus 0.4em\relax NY: Springer, 1997.

\bibitem{deprima1974range}
C.~R. DePrima and C.~R. Johnson, ``The range of {$A^{-1}A^{*}$} in {GL(n,
  C)},'' \emph{Linear Algebra and its Applications}, vol.~9, pp. 209--222,
  1974.

\bibitem{furtado2001spectral}
S.~Furtado and C.~R. Johnson, ``Spectral variation under congruence,''
  \emph{Linear and Multilinear Algebra}, vol.~49, no.~3, pp. 243--259, 2001.

\bibitem{johnson2001generalization}
C.~R. Johnson and S.~Furtado, ``A generalization of {S}ylvester's law of
  inertia,'' \emph{Linear Algebra and its Applications}, vol. 338, no. 1-3, pp.
  287--290, 2001.

\bibitem{zhang2015matrix}
F.~Zhang, ``A matrix decomposition and its applications,'' \emph{Linear and
  Multilinear Algebra}, vol.~63, no.~10, pp. 2033--2042, 2015.

\bibitem{wang2020phases}
D.~Wang, W.~Chen, S.~Z. Khong, and L.~Qiu, ``On the phases of a complex
  matrix,'' \emph{Linear Algebra and its Applications}, vol. 593, pp. 152--179,
  2020.

\bibitem{furtado2003spectral}
S.~Furtado and C.~R. Johnson, ``Spectral variation under congruence for a
  nonsingular matrix with 0 on the boundary of its field of values,''
  \emph{Linear Algebra and Its Applications}, vol. 359, no. 1-3, pp. 67--78,
  2003.

\bibitem{chen2024phase}
W.~Chen, D.~Wang, S.~Z. Khong, and L.~Qiu, ``A phase theory of multi-input
  multi-output linear time-invariant systems,'' \emph{SIAM Journal on Control
  and Optimization}, vol.~62, no.~2, pp. 1235--1260, 2024.

\bibitem{wang2023phases}
D.~Wang, X.~Mao, W.~Chen, and L.~Qiu, ``On the phases of a semi-sectorial
  matrix and the essential phase of a laplacian,'' \emph{Linear Algebra and its
  Applications}, vol. 676, pp. 441--458, 2023.

\bibitem{postlethwaite1981principal}
I.~Postlethwaite, J.~M. Edmunds, and A.~G. MacFarlane, ``Principal gains and
  principal phases in the analysis of linear multivariable feedback systems,''
  \emph{IEEE Transactions on Automatic Control}, vol.~26, no.~1, pp. 32--46,
  1981.

\bibitem{owens1984numerical}
D.~H. Owens, ``The numerical range: a tool for robust stability studies?''
  \emph{Systems \& Control Letters}, vol.~5, no.~3, pp. 153--158, 1984.

\bibitem{anderson1988hilbert}
B.~D. Anderson and M.~Green, ``Hilbert transform and gain/phase error bounds
  for rational functions,'' \emph{IEEE Transactions on Circuits and Systems},
  vol.~35, no.~5, pp. 528--535, 1988.

\bibitem{haddad1992there}
W.~M. Haddad and D.~Bernstein, ``Is there more to robust control theory than
  small gain,'' in \emph{Proceedings of the 1992 American Control Conference},
  1992, pp. 83--84.

\bibitem{chen1998multivariable}
J.~Chen, ``Multivariable gain-phase and sensitivity integral relations and
  design trade-offs,'' \emph{IEEE Transactions on Automatic Control}, vol.~43,
  no.~3, pp. 373--385, 1998.

\bibitem{chen2019phase}
W.~Chen, D.~Wang, S.~Z. Khong, and L.~Qiu, ``Phase analysis of {MIMO LTI}
  systems,'' in \emph{2019 IEEE 58th Conference on Decision and Control
  (CDC)}.\hskip 1em plus 0.5em minus 0.4em\relax IEEE, 2019, pp. 6062--6067.

\bibitem{mao2021phase}
X.~Mao, W.~Chen, and L.~Qiu, ``Phases of discrete-time {LTI} multivariable
  systems,'' \emph{Automatica}, vol. 142, p. 110311, 2022.

\bibitem{chen2020phase}
C.~Chen, D.~Zhao, W.~Chen, S.~Z. Khong, and L.~Qiu, ``Phase of nonlinear
  systems,'' \emph{Submitted. Preprint: arXiv:2012.00692}, 2020.

\bibitem{ballantine1975accretive}
C.~S. Ballantine and C.~R. Johnson, ``Accretive matrix products,'' \emph{Linear
  and Multilinear Algebra}, vol.~3, no.~3, pp. 169--185, 1975.

\bibitem{khong2018robust}
S.~Z. Khong, I.~R. Petersen, and A.~Rantzer, ``Robust stability conditions for
  feedback interconnections of distributed-parameter negative imaginary
  systems,'' \emph{Automatica}, vol.~90, pp. 310--316, 2018.

\bibitem{lancaster1985thetheory}
P.~Lancaster and M.~Tismenetsky, \emph{The Theory of Matrices}, 2nd~ed.\hskip
  1em plus 0.5em minus 0.4em\relax Orlando, FL: Academic press, 1985.

\bibitem{gohberg2006indefinite}
I.~Gohberg, P.~Lancaster, and L.~Rodman, \emph{Indefinite Linear Algebra and
  Applications}.\hskip 1em plus 0.5em minus 0.4em\relax Basel: Birkh{\"a}user,
  2006.

\bibitem{horn1954eigenvalues}
A.~Horn, ``On the eigenvalues of a matrix with prescribed singular values,''
  \emph{Proceedings of the American Mathematical Society}, vol.~5, no.~1, pp.
  4--7, 1954.

\bibitem{cvx}
M.~Grant and S.~Boyd, ``{CVX}: Matlab software for disciplined convex
  programming, version 2.1,'' \url{http://cvxr.com/cvx}, Mar. 2014.

\bibitem{grant2008graph}
------, ``Graph implementations for nonsmooth convex programs,'' in
  \emph{Recent Advances in Learning and Control}, ser. Lecture Notes in Control
  and Information Sciences, V.~Blondel, S.~Boyd, and H.~Kimura, Eds.\hskip 1em
  plus 0.5em minus 0.4em\relax Springer-Verlag Limited, 2008, pp. 95--110.

\bibitem{higham2008functions}
N.~J. Higham, \emph{Functions of Matrices: Theory and Computation}.\hskip 1em
  plus 0.5em minus 0.4em\relax Philadelphia, PA: SIAM, 2008.

\bibitem{johnson1974field}
C.~R. Johnson, ``The field of values and spectra of positive definite
  multiples,'' \emph{Journal of Research of the Notional Bureau of Standards -
  B. Mathematical Sciences}, vol.~78, no.~4, pp. 197--8, 1974.

\bibitem{marshall2011inequalities}
A.~W. Marshall, I.~Olkin, and B.~C. Arnold, \emph{Inequalities: Theory of
  Majorization and Its Applications}, 2nd~ed.\hskip 1em plus 0.5em minus
  0.4em\relax New York: Springer, 2011.

\end{thebibliography}

\end{document}